\documentclass[onefignum, onetabnum]{siamart171218}
\usepackage{amsmath, amssymb, amsopn}
\usepackage{array}
\usepackage{amscd}
\usepackage{epsfig}
\usepackage{color}
\usepackage{comment}
\usepackage{setspace}
\usepackage{textcomp}
\usepackage{stmaryrd}
\usepackage{multirow}
\usepackage{diagbox}
\usepackage{booktabs}
\usepackage{todonotes}
\usepackage{subcaption}
\usepackage{placeins}

\numberwithin{equation}{section}

\newtheorem{prop}[theorem]{Proposition}

\newtheorem{example}[theorem]{Example}
\newtheorem{rmk}[theorem]{Remark}

\DeclareMathOperator{\Div}{div}

\DeclareMathOperator{\Grad}{\nabla}

\DeclareMathOperator{\ssum}{\textstyle \sum}

\newcommand{\triang}{\mathcal{T}}
\newcommand{\inner}[2]{\langle #1, #2 \rangle}
\newcommand{\foralls}{\forall \,}
\newcommand{\ddt}{\frac{\rm d}{\rm dt}}
\newcommand{\ds}{\, \mathrm{d} s}
\newcommand{\ea}{h}

\newcommand{\esssup}{\operatorname{ess \, sup}}

\newcommand{\half}{\frac{1}{2}}
\newcommand{\R}{{\mathbb R}}
\newcommand{\X}{{\mathbb X}}
\newcommand{\strain}{\varepsilon}

\headers{Mixed finite elements for MPET}{J. J. Lee, E. Piersanti, K.-A. Mardal, and M. E. Rognes}

\title{A mixed finite element method for nearly incompressible multiple-network poroelasticity \thanks{Submitted to the editors \today. \funding{The work of J.~J.~Lee has been supported by the European Research Council under the European Union's Seventh Framework Programme (FP7/2007-2013) ERC grant agreement 339643. The work of M.~E.~Rognes and K.-A.Mardal have been supported by the Research Council of Norway under the FRINATEK Young Research Talents Programme through project \#250731/F20 (Waterscape). E.~Piersanti is a doctoral fellow in the Simula-UCSD-University of Oslo Research and PhD training (SUURPh) program, an international collaboration in computational biology and medicine funded by the Norwegian Ministry of Education and Research.}}}

\author{J. J. Lee\thanks{Institute for Computational Engineering and Sciences, The University of Texas at Austin, 201 E. 24th Street, POB 4.102, Austin, Texas 78712, USA (\email{johnlee04@gmail.com})}
  \and E. Piersanti\thanks{Simula Research Laboratory, P. O. Box 134, 1325 Lysaker, Norway (\email{eleonora@simula.no})}
  \and K.-A. Mardal\thanks{Department of Mathematics, University of Oslo, P. O. Box 1053 Blindern, 0316 Oslo, Norway and Simula Research Laboratory, P. O. Box 134, 1325 Lysaker, Norway (\email{kent-and@simula.no})}
  \and M. E. Rognes\thanks{Simula Research Laboratory, P. O. Box 134, 1325 Lysaker, Norway (\email{meg@simula.no})}}

\begin{document}

\maketitle

\begin{abstract}
  In this paper, we present and analyze a new mixed finite element
  formulation of a general family of quasi-static multiple-network
  poroelasticity (MPET) equations. The MPET equations describe flow
  and deformation in an elastic porous medium that is permeated by
  multiple fluid networks of differing characteristics. As such, the
  MPET equations represent a generalization of Biot's equations, and
  numerical discretizations of the MPET equations face similar
  challenges. Here, we focus on the nearly incompressible case for
  which standard mixed finite element discretizations of the MPET
  equations perform poorly. Instead, we propose a new mixed finite
  element formulation based on introducing an additional total
  pressure variable. By presenting energy estimates for the continuous
  solutions and \emph{a priori} error estimates for a family of
  compatible semi-discretizations, we show that this formulation is
  robust in the limits of incompressibility, vanishing storage
  coefficients, and vanishing transfer between networks. These
  theoretical results are corroborated by numerical experiments. Our
  primary interest in the MPET equations stems from the use of these
  equations in modelling interactions between biological fluids and
  tissues in physiological settings. So, we additionally present
  physiologically realistic numerical results for blood and tissue
  fluid flow interactions in the human brain.
\end{abstract}

\begin{keywords}
  multiple-network poroelasticity, mixed finite element,
  incompressible, cerebral fluid flow
\end{keywords}

\begin{AMS}
65M12, 65M15, 65M60, 92C10
\end{AMS}

\section{Introduction}

In this paper, we consider a family of quasi-static multiple-network
poroelasticity (MPET\footnote{The abbreviation MPET stems from the
  term multiple-network poroelastic theory as used by
  e.g.~\cite{TullyVentikos2011}. Here, we instead refer to the
  multiple-network poroelasticity equations but keep the abbreviation
  for the sake of convenience.}) equations reading as follows: for a
given number of networks $A \in \mathbb{N}$, find the displacement $u$
and the network pressures $p_{j}$ for $j = 1, \dots, A$ such that
\begin{subequations}
  \label{eq:mpet}
  \begin{align}
    \label{eq:mpet:1}
    - \Div C \strain(u) + \ssum_{j} \alpha_{j} \Grad p_{j} &= f, \\
    \label{eq:mpet:2}
    c_{j} \dot{p}_{j} + \alpha_{j} \Div \dot{u} - \Div K_{j} \Grad p_{j} + S_{j} &= g_{j}, \qquad 1 \le j \le A,
  \end{align}
\end{subequations}
where $u = u(x, t)$ and $p_j = p_j(x, t)$, $1\le j \le A$ for $x \in
\Omega \subset \R^{d}$ ($d = 1, 2, 3$) and for $t \in [0, T]$.

In our context, \eqref{eq:mpet} originates from balance of mass and
momentum in a porous, linearly elastic medium permeated by $A$
segregated viscous fluid networks. The operators and parameters are as
follows: $C$ is the elastic stiffness tensor, each network $j$ is
associated with a Biot-Willis coefficient $\alpha_{j} \in (0, 1]$,
  storage coefficient $c_{j} \geq 0$, and hydraulic conductivity
  tensor $K_{j} = \kappa_{j}/\mu_{j} > 0$ (where $\kappa_{j}$ and
  $\mu_{j}$ represent the network permeability and the network fluid
  viscosity, respectively). In~\eqref{eq:mpet:1}, $\Grad$ denotes the
  gradient, $\strain$ is the symmetric (row-wise) gradient, $\Div$
  denotes the row-wise divergence. In~\eqref{eq:mpet:2}, $\Grad$ and
  $\Div$ are the standard gradient and divergence operators, and the
  superposed dot denotes the time derivative. Further, $f$ represents
  a body force and $g_{j}$ represents sources in network $j$ for $j =
  1, \dots, A$, while $S_{j}$ represents transfer terms out of network
  $j$.

In this paper, we consider the case of an isotropic stiffness tensor
for which
\begin{equation}
  \label{eq:def:isotropic}
  C \strain(u) = 2 \mu \strain(u) + \lambda \Div u I
\end{equation}
where $\mu, \lambda$ are the standard non-negative Lam\'e parameters
and $I$ denotes the identity tensor. Moreover, we will consider the
case where the transfer terms $S_j$, quantifying the transfer out of
network $j$ into the other fluid networks, are proportional to
pressure differences between the networks. More precisely, we assume
that $S_j$ takes the form:
\begin{align}
  \label{eq:def:transfer}
  S_j = S_j(p_1, \dots, p_A) = \ssum_{i=1}^A \xi_{j \leftarrow i} (p_j - p_i), 
\end{align}
where $\xi_{j \leftarrow i}$ are non-negative transfer coefficients
for $i, j = 1, \dots, A$. We will also assume that these transfer
coefficients are symmetric in the sense that $\xi_{j \leftarrow i} =
\xi_{i \leftarrow j}$, and note that $\xi_{j \leftarrow j}$ is
arbitrary.
  
The MPET equations have an abundance of both geophysical and
biological applications. In the case $A = 1$,~\eqref{eq:mpet} reduces
to the well-known quasi-static Biot equations. While the Biot
equations have been studied extensively, see e.g.~\cite{Showalter2000,
  MuradEtAl1996,PhillipsWheeler2007, AguilarEtAl2008,
  OyarzuaRuizBaier2016, Lee2017, Yi2017}; to the best of our
knowledge, the general multiple-network poroelasticity equations have
received much less attention, especially from the numerical
perspective. The case $A = 2$ is known as the Barenblatt-Biot model,
and we note that Showalter and Momken~\cite{ShowalterMomken2002}
present an existence analysis for this model, while Nordbotten and
co-authors~\cite{NordbottenEtAl2010} present an \emph{a posteriori}
error analysis for an approximation of a static Barenblatt-Biot
system.

Our interest in the multiple-network poroelasticity equations
primarily stems from the use of these equations in modelling
interactions between biological fluids and tissue in physiological
settings. As one example, Tully and Ventikos~\cite{TullyVentikos2011}
considers~\eqref{eq:mpet} with four different networks ($A = 4$) to
model fluid flows, network pressures and elastic displacement in brain
tissue. The fluid networks represent the arteries, the
arterioles/capillaries, the veins and the interstitial fluid-filled
extracellular space, each network with e.g.~a different permeability
$\kappa_j$ and different transfer coefficients $\xi_{j \leftarrow i}$.

A particularly important motivation for the current work is the
recently proposed theory of the glymphatic system which describes a
new mechanism for waste clearance in the human
brain~\cite{iliff2012paravascular,jessen2015glymphatic,AbbottEtAl2018}. This
mechanism is proposed to take the form of a convective flow of
water-like fluid through (a) spaces surrounding the cerebral
vasculature (paravascular spaces) and (b) through the extracellular
spaces, driven by a hydrostatic pressure gradient between the arterial
and venous compartments. Compared to diffusion only, such a convective
flow would lead to enhanced transport of solutes through the brain
parenchyma and, in particular, contribute to clearance of metabolic
waste products such as amyloid beta. The accumulation of amyloid beta
frequently seen in patients with Alzheimer's disease is as such seen
as a malfunction of the glymphatic system. In this context, the
original system of \cite{TullyVentikos2011} represents a macroscopic
model of interaction between the different fluid networks in the
brain.

Discretization of Biot's equations is known to be challenging, in
particular because of so-called poroelastic locking. Poroelastic
locking has two main characteristics: 1) underestimation of the solid
deformation if the material is close to being incompressible and 2)
nonphysical pressure oscillations, in particular in the areas close to
jumps in the permeabilities or to the boundary. Several recent (and
not so recent) studies, see e.g.~\cite{PhillipsWheeler2007,
  berger2015stabilized, bause2017space, hu2017nonconforming,
  rodrigo2017new, Yi2017}, focus on a three-field formulation of
Biot's model, involving the elastic displacement, fluid pressure and
fluid velocity.  Four-field formulations where also the elasticity
equation is in mixed form, designed to provide robust numerical
methods for nearly incompressible materials, have also been
studied~\cite{Yi2014, KorsaweStarke2005,Lee2016}.

In biological tissues, any jumps in the permeability parameters are
typically small in contrast to geophysical applications. The challenge
in the biomedical applications is rather that the tissues in our body
mostly consist of water and as such should be close to be
incompressible (for short time-scales and normal physiological
pressures). Therefore, it may be crucial for accurate modeling of the
interaction of the different network pressures in~\eqref{eq:mpet} to
allow for an elastic material that is almost incompressible and/or
with (nearly) vanishing storage coefficients, i.e.~for $1 \ll \lambda
< +\infty$ and $0 < c_j \ll 1$ in~\eqref{eq:mpet}. Standard two-field
mixed finite element discretizations of the Biot model, approximating
the displacement and the fluid pressure only using Stokes-stable
elements, are well-known to perform poorly in the incompressible
limit, see e.g.~\cite{Lee2017} and references therein. Moreover, we
can easily demonstrate a suboptimal convergence rate for the
corresponding standard mixed finite element discretization of the MPET
equations, see Example~\ref{ex:mpet:standard} below. On the other
hand, two-field approximations are computationally inexpensive
compared to three-field approximations in the sense that only one
unknown, the network pressure, is involved in each network.

\begin{example}
  \label{ex:mpet:standard}
  To illustrate poor performance of a standard mixed finite element
  discretization of the MPET equations~\eqref{eq:mpet} in the nearly
  incompressible case, we consider a variant of the smooth test case
  presented by~\cite[Section 7.1]{Yi2017}. Let $\Omega = [0, 1]^2
  \subset \R^2$, take $T = 0.5$, and consider the quasi-static
  multiple-network poroelasticity equations~\eqref{eq:mpet} with $A =
  2$, $c_j = 1.0$, $K_j = 1.0$, $\alpha_j = 1.0$, and $S_j = 0$ for $j
  = 1, 2$. Moreover, we let $E = 1.0$ and $\nu = 0.49999$ for
  \begin{equation*}
    \mu = \frac{E}{2 (1 + \nu)} \approx \frac{1}{3}, \quad
    \lambda = \frac{\nu E}{(1 - 2 \nu)(1 + \nu)} \approx 16\,666.
  \end{equation*}
  To discretize~\eqref{eq:mpet}, we consider a Crank-Nicolson
  discretization in time and a standard mixed finite element
  discretization in space in this example. More precisely, we
  approximate the displacement $u$ using continuous piecewise
  quadratic vector fields (and denote the approximation by $u_h$) and
  the fluid pressures $p_j$ for $j = 1, 2$ using continuous piecewise
  linears defined relative to a uniform mesh of $\Omega$ of mesh size
  $h$. As exact solutions, we let
  \begin{equation*}
    u((x_0, x_1), t) = t \begin{pmatrix}
       (\sin(2 \pi x_1) (-1 + \cos (2 \pi x_0)) + \frac{1}{\mu + \lambda} \sin (\pi x_0) \sin (\pi x_1) ) \\
       (\sin(2 \pi x_0) (1 - \cos (2 \pi x_1)) + \frac{1}{\mu + \lambda} \sin (\pi x_0) \sin (\pi x_1) )
    \end{pmatrix},
  \end{equation*}
  and
  \begin{equation*}
    p_j((x_0, x_1, t)) = - j t \sin (\pi x_0) \sin ( \pi x_1).
  \end{equation*}
  The resulting approximation errors for $u(T)$ in the $L^2(\Omega)$
  and $H^1(\Omega)$ norms are listed in Table~\ref{tab:mpet:standard}
  for a series of meshes generated by nested uniform refinements,
  together with the corresponding rates of convergence.  We observe
  that the convergence rates are one order sub-optimal for this choice
  of spatial discretization.
  \begin{table}[h]
    \centering
    \begin{tabular}{lcccc}
    \toprule
    $h$ & $\| u(T) - u_h(T) \|$ & Rate & 
    $\| u(T) - u_h(T) \|_{H^1}$ & Rate \\
    \midrule
    $H$ & 0.169 & & 2.066 & \\
    $H/2$ & 0.040 & 2.09 & 0.980 & 1.08 \\
    $H/4$ & 0.010 & 2.04 & 0.480 & 1.03 \\
    $H/8$ & 0.002 & 2.03 & 0.235 & 1.03 \\
    $H/16$ & 0.001 & 2.09 & 0.110 &1.10 \\
    \midrule
    Optimal & & 3 & & 2 \\
    \bottomrule
    \vspace{0.5em}
  \end{tabular}
  \caption{Approximation errors in the $L^2$ ($\| \cdot \|$)- and
    $H^1$ ($\| \cdot \|_{H^1}$)-norms and associated convergence rates
    for a standard mixed finite element discretization for a smooth
    manufactured solution test case for a nearly incompressible
    material (Example~\ref{ex:mpet:standard}). $H$ corresponds to a
    uniform mesh constructed by dividing the unit square into $4
    \times 4$ squares and dividing each square by a diagonal.}
  \label{tab:mpet:standard}
  \end{table}
\end{example}

The primary objective of this paper is to propose and analyze a new
variational formulation and a corresponding spatial discretization of
the MPET equations that are robust with respect to a nearly
incompressible poroelastic matrix; i.e.~the implicit constants in the
error estimates are uniformly bounded for arbitrarily large $\lambda >
0$. To this end, we introduce a formulation with one additional scalar
field unknown. For the MPET equations~\eqref{eq:mpet} with potentially
multiple networks, the additional computational cost is thus
small. Instead of taking the "solid pressure" $\lambda \Div u$ as a
new unknown, we take the total pressure, which is defined as a
weighted sum of the network pressures and the solid pressure, as the
new unknown. Such a formulation has previously been shown to be
advantageous in the context of parameter-robust preconditioners for
the Biot model~\cite{LeeEtAl2017}. Here, we focus on stability and
error estimates of the total pressure formulation for the more general
MPET equations. The construction of preconditioners for the MPET
equations will be addressed in a forthcoming paper.

Our new theoretical results include an energy estimate for the
continuous variational formulation that is robust in the relevant
parameter limits, in particular, that is uniform in the Lam\'e
parameter $\lambda$, storage coefficients $c_j$ for $j = 1, \dots, A$,
and transfer coefficients $\xi_{j \leftarrow i}$ for $i, j = 1, \dots,
A$, and a robust \emph{a priori} error estimate for a class of
compatible semi-discretizations of the new formulation. These
theoretical results are supported by numerical experiments. Finally,
we also present new numerical MPET simulations modelling blood and
tissue fluid interactions in a physiologically realistic human brain.

This paper is organized as follows. Section~\ref{sec:notation}
presents notation and general preliminaries. In
Section~\ref{sec:mpet:qs}, we introduce a total-pressure-based
variational formulation~\eqref{eq:mpet:vf} for the quasi-static MPET
equations~\eqref{eq:mpet}, together with a robust energy estimate in
Theorem~\ref{thm:mpet:qs:ee}. We continue in
Section~\ref{sec:mpet:qs:sd} by proposing a general class of
compatible semi-discretizations~\eqref{eq:quasistatic:fem} of this
formulation, and estimate the \emph{a priori} discretization errors in
Proposition~\ref{prop:eh} and the semi-discrete errors for a specific
choice of finite element spaces in
Theorem~\ref{thm:mpet:qs:taylor-hood} and
Proposition~\ref{prop:L2error:pj}. These theoretical results are
corroborated by synthetic numerical convergence experiments in
Section~\ref{sec:numerics:convergence}. In
Section~\ref{sec:numerics:brain}, we present a more physiologically
realistic numerical experiment using a 4-network MPET model to
investigate blood and tissue fluid flow in the human brain. Some
conclusions and directions of future research are highlighted in
Section~\ref{sec:conclusion}.

\section{Notation and preliminaries}
\label{sec:notation}

Throughout this paper we use $X \lesssim Y$ to denote the inequality
$X \leq C Y$ with a generic constant $C > 0$ which is independent of
mesh sizes. If needed, we will write $C$
explicitly in inequalities but it can vary across expressions.

\subsection{Sobolev spaces}
Let $\Omega$ be a bounded polyhedral domain in $\R^d$ ($d=1, 2$, or
$3$) with boundary $\partial \Omega$. We let $L^2(\Omega)$ be the set
of square-integrable real-valued functions on $\Omega$. The inner
product of $L^2(\Omega)$ and the induced norm are denoted by
$\inner{\cdot}{ \cdot}$ and $\| \cdot \|$, respectively. For a
finite-dimensional inner product space $\X$, typically $\X = \R^d$,
let $L^2(\Omega; \X)$ be the space of $\X$-valued functions such that
each component is in $L^2(\Omega)$. The inner product of $L^2(\Omega;
\X)$ is naturally defined by the inner product of $\X$ and
$L^2(\Omega)$, so we use the same notation $\inner{\cdot}{\cdot}$ and
$\| \cdot \|$ to denote the inner product and norm on $L^2(\Omega;
\X)$. For a non-negative real-valued function on $\Omega$ (or
symmetric positive semi-definite tensor-valued function on $\Omega$)
$w$, we also introduce the short-hand notations
\begin{equation}
  \label{eq:def:weighted:L2:norm}
  \inner{u}{v}_{w} = \inner{w u}{v},
  \quad \| u \|_{w}^2 = \inner{u}{u}_{w} ,
\end{equation}
noting that the latter is a norm only when $w$ is strictly positive
a.e.~on $\Omega$ (or is positive definite a.e.~on $\Omega$).

For a non-negative integer $m$, $H^m(\Omega)$ denotes the standard
Sobolev spaces of real-valued functions based on the $L^2$-norm, and
$H^m(\Omega; \X)$ is defined similarly based on $L^2(\Omega; \X)$. To
avoid confusion with the weighted $L^2$-norms
cf.~\eqref{eq:def:weighted:L2:norm} we use $\| \cdot \|_{H^m}$ to
denote the $H^m$-norm (both for $H^m(\Omega)$ and $H^m(\Omega;
\X)$). For $m \geq 1$, we use $H^m_{0, \Gamma}(\Omega)$ to denote the
subspace of $H^m(\Omega)$ with vanishing trace on $\Gamma \subset
\partial \Omega$, and $H^m_{0, \Gamma}(\Omega; \X)$ is defined
similarly~\cite{Evans1998}. For $\Gamma = \partial \Omega$, we write
$H^m_0(\Omega)$ and analogously $H^m_{0}(\Omega; \X)$.

\subsection{Spaces involving time}

We will consider an interval $[0, T]$, $T > 0$. For a reflexive Banach
space $\mathcal{X}$, let $C^0 ([0, T] ; \mathcal{X})$ denote the set
of functions $f : [0, T] \rightarrow \mathcal{X}$ that are continuous
in $t \in [0, T]$. For an integer $m \geq 1$, we define
\begin{equation*}
  C^m ([0, T]; \mathcal{X}) = \{ f \, | \, \partial^{i}f/\partial t^{i} \in C^0([0, T];\mathcal{X}), \, 0 \leq i \leq m \},
\end{equation*}
where $\partial^i f/\partial t^i$ is the $i$-th time derivative in the
sense of the Fr\'echet derivative in $\mathcal{X}$ (see
e.g.~\cite{Yosida1980}).

For a function $f : [0, T] \rightarrow \mathcal{X}$, we define the
space-time norm
\begin{equation*}
  \| f \|_{L^r([0, T]; \mathcal{X})} =
  \begin{cases}
    \left( \int_0^T \| f(s) \|_\mathcal{X}^r \ds \right)^{1/r}, \quad 1 \leq r < \infty, \\
    \esssup_{t \in [0, T]} \| f (t) \|_\mathcal{X}, \quad r = \infty.
  \end{cases}
\end{equation*}
We define the space-time Sobolev spaces $W^{k,r}([0, T]; \mathcal{X})$
for a non-negative integer $k$ and $1 \leq r \leq \infty$ as the
closure of $C^k ([0, T]; \mathcal{X})$ with the norm $\| f
\|_{W^{k,r}([0, T];\mathcal{X})} = \sum_{i=0}^k \| \partial^i f /
\partial t^i \|_{L^r([0, T]; \mathcal{X})}$.

\subsection{Finite element spaces}

Let $\triang_h$ be an admissible, conforming, simplicial tessellation
of the domain $\Omega$. For any integer $k \geq 1$, we let
$\mathcal{P}_k(\triang_h)$ denote the space of continuous piecewise
polynomials of order $k$ defined relative to $\triang_h$, and
$\mathcal{P}_k^d(\triang_h)$ as the space of $d$-tuples with
components in $\mathcal{P}_k$. We will typically omit the reference to
$\triang_h$ when context allows. We let $\mathring{\mathcal{P}_k}$
denote the restriction of these piecewise polynomial spaces to conform
with given essential homogeneous boundary conditions.

\subsection{Parameter values}

Based on physical considerations and typical applications, we will
make the following assumptions on the material parameter
values. First, we assume that the Biot-Willis coefficients $\alpha_j
\in (0, 1]$, $j = 1, \dots, A$, and the storage coefficients $c_j > 0$
  are constant in time for $j = 1, \dots, A$. In the analysis, we will
  pay particular attention to robustness of estimates with respect to
  arbitrarily large $\lambda$ and arbitrarily small (but not
  vanishing) $c_j$'s. We also comment on the case $c_j = 0$ in
  Remark~\ref{rmk:cj0}.

We will assume that the hydraulic conductivities $K_j$ are constant in
time, but possibly spatially-varying and that these satisfy standard
ellipticity constraints: i.e.~there exist positive constants $K_j^{-}$
and $K_j^{+}$ such that
\begin{equation*}
  K_j^{-} \leq K_j(x) \leq K_j^{+} \quad \foralls x \in \Omega.
\end{equation*}
We assume that the transfer coefficients $\xi_{j \leftarrow i}$ are
constant in time and non-negative: i.e.~$\xi_{j \leftarrow i}(x) \geq
0$ for $1 \leq i, j \leq A$, $x \in \Omega$.

\subsection{Boundary conditions}

We will consider~\eqref{eq:mpet} augmented by the following standard
boundary conditions. First, we assume that the boundary decomposes in
two parts: $\partial \Omega = \Gamma_D \cup \Gamma_N$ with $\Gamma_D
\cap \Gamma_N = \emptyset$ and $|\Gamma_D|, |\Gamma_N| > 0$ where
$|\Gamma|$ is the Lebesgue measure of $\Gamma$.  We use $n$ to denote
the outward unit normal vector field on $\partial \Omega$.  Relative
to this partition, we consider the homogeneous boundary conditions
\begin{subequations}
  \label{eq:mpet:bcs}
  \begin{align}
    u &= 0 \quad \text{ on } \Gamma_D , \\
    C \strain(u) \cdot n &= 0 \quad \text{ on } \Gamma_N , \\
    p_j &= 0 \quad \text{ on } \partial \Omega \quad \text{ for } j = 1, \dots, A . 
  \end{align}
\end{subequations}
The subsequent formulation and analysis can easily be extended to
cover inhomogeneous and other types of boundary conditions.

\subsection{Key inequalities}

For the space $V = H^1_{0, \Gamma_D}(\Omega)$, Korn's
inequality~\cite[p.~288]{Braess2001} holds; i.e. there exists a
constant $C > 0$ depending only on $\Omega$ and $\Gamma_D$ such that
\begin{align}
  \label{eq:mpet:korn}
  \|u\| \le C \|\strain(u) \| \quad \foralls u \in V .  
\end{align}

Furthermore, for the combination of spaces $V$ and $Q_0 =
L^2(\Omega)$, the following (continuous Stokes) inf-sup condition
holds: there exists a constant $C > 0$ depending only on $\Omega$ and
$\Gamma_D$ such that
\begin{align} 
\label{eq:mpet:infsup} 
\sup_{u \in V} \frac{\inner{\Div u}{q}}{\|u\|_{H^1}} \ge C \| q \| \quad \foralls q \in L^2(\Omega) .   
\end{align} 
Our discretization schemes will also satisfy corresponding discrete
versions of Korn's inequality and the inf-sup condition with constants
independent of the discretization.

\subsection{Initial conditions}

The MPET equations~\eqref{eq:mpet} must also be complemented by
appropriate initial conditions. In particular, in agreement with the
assumption that $c_j > 0$ for $j = 1, \dots, A$, we assume that
initial conditions are given for all $p_j$:
\begin{equation}
  \label{eq:mpet:ics}
    p_j(x, 0) = p_j^0(x), \quad x \in \Omega, \quad j = 1, \dots, A .
\end{equation}
Given such $p_j^0$, we note that we may compute $u(x, 0) = u^0(x)$
from~\eqref{eq:mpet:1}, which in particular yields a $\Div u(x, 0) =
\Div u^0(x)$ for $x \in \Omega$. In the following, we will assume that
any initial conditions given are compatible in the sense described
here.

\section{A new formulation for multiple-network poroelasticity}
\label{sec:mpet:qs}

In this section, we introduce a new variational formulation for the
quasi-static multiple-network poroelasticity equations targeting the
incompressible and nearly incompressible regime. Inspired
by~\cite{OyarzuaRuizBaier2016, LeeEtAl2017}, we introduce an
additional variable, namely the \emph{total pressure}. In the
subsequent subsections, we present the augmented governing equations,
introduce a corresponding variational formulation, and demonstrate the
robustness of this formulation via an energy estimate.

\subsection{Governing equations introducing the total pressure}

Let $u$ and $p_j$ for $j = 1, \dots, A$ be solutions
of~\eqref{eq:mpet} with boundary conditions given
by~\eqref{eq:mpet:bcs}, initial conditions given
by~\eqref{eq:mpet:ics} and recall the isotropic stiffness tensor
assumption, cf.~\eqref{eq:def:isotropic}. Additionally, we now
introduce the total pressure $p_0$ defined as
\begin{equation}
  \label{eq:def:p_0}
  p_0 = \lambda \Div u - \ssum_{j=1}^A \alpha_j p_j. 
\end{equation}
Defining $\alpha_0 = 1$ for the purpose of short-hand, and
rearranging, we thus have that
\begin{equation}
  \label{eq:divu:as:p0}
  \Div u = \lambda^{-1} \ssum_{i=0}^A \alpha_i p_i .
\end{equation}
For simplicity, we denote $\alpha = (\alpha_0, \alpha_1, \dots,
\alpha_A)$ and $p = (p_0, p_1, \dots, p_A)$, and we can thus write
\begin{equation*}
  \ssum_{i=0}^A \alpha_i p_i = \alpha \cdot p 
\end{equation*}
in the following.

Inserting~\eqref{eq:divu:as:p0} and its time-derivative
into~\eqref{eq:mpet:2}, we obtain an augmented system of quasi-static
multiple-network poroelasticity equations: for $t \in (0, T]$, find
  the displacement vector field $u$ and the pressure scalar fields
  $p_i$ for $i = 0, \dots, A$ such that
\begin{subequations}
  \label{eq:mpet:tp}
  \begin{align}
    \Div u - \lambda^{-1} \alpha \cdot p &= 0, \\ 
    - \Div \left ( 2 \mu \strain (u) + p_0 I \right ) &= f, \\
    c_j \dot{p}_j + \alpha_j \lambda^{-1} \alpha \cdot \dot{p} - \Div (K_j \nabla p_j) + S_j &= g_j \quad j = 1, \dots, A.
  \end{align}
\end{subequations} 
We note that $p_0(x, 0)$ can be computed from~\eqref{eq:mpet:ics}
and~\eqref{eq:def:p_0}.
\begin{rmk}
 \label{rmk:1}
 In the limit $\lambda = \infty$, the equations for the displacement
 $u$ and total pressure $p_0$, and the network pressures $p_i$
 decouple, and~\eqref{eq:mpet:tp} reduces to a Stokes system for $(u,
 p_0)$ and a system of parabolic equations for $p_j$:
\begin{subequations}
  \begin{align*}
    - \Div \left ( 2 \mu \strain (u) + p_0 I \right ) &= f, \\
    \Div u  &= 0, \\ 
    c_j \dot{p}_j - \Div (K_j \nabla p_j) + S_j &= g_j \quad j = 1, \dots, A .
  \end{align*}
\end{subequations} 
\end{rmk}

We next present and study a continuous variational formulation based
on the total pressure formulation~\eqref{eq:mpet:tp} of the
quasi-static multiple-network poroelasticity equations.

\subsection{Variational formulation}

With reference to the notation for domains and Sobolev spaces as
introduced in Section~\ref{sec:notation}, let
\begin{equation}
V = H^1_{0, \Gamma_D} (\Omega; \R^d), \quad Q_0 = L^2(\Omega), \quad Q_j =
H^1_0(\Omega) \quad j = 1, \dots, A. 
\end{equation}
Also denote $Q = Q_0 \times Q_1 \times \dots \times Q_A$.

Multiplying~\eqref{eq:mpet:tp} by test functions and integrating by
parts with boundary conditions given by~\eqref{eq:mpet:bcs} and
initial conditions given by~\eqref{eq:mpet:ics} yield the following
variational formulation: given compatible $u^0$ and $p_j^0$, $f$ and
$g_j$ for $j = 1, \dots, A$, find $u \in C^1([0, T]; V)$ and $p_i \in
C^1([0, T], Q_i)$ for $i = 0, \dots, A$ such that
\begin{subequations}
  \label{eq:mpet:vf}
  \begin{alignat}{2}
  \label{eq:mpet:vf:1}
    \inner{2 \mu \strain(u)}{\strain(v)} + \inner{p_0}{\Div v} &= \inner{f}{v} &&\quad \foralls v \in V, \\
    \label{eq:mpet:vf:2}
    \inner{\Div u}{q_0} - \inner{\lambda^{-1} \alpha \cdot p}{q_0} &= 0 &&\quad \foralls q_0 \in Q_0, \\
    \inner{c_j \dot{p}_j + \alpha_j \lambda^{-1} \alpha \cdot \dot{p}  + S_j}{q_j} + \inner{K_j \Grad p_j}{\Grad q_j} &= \inner{g_j}{q_j}
    &&\quad \foralls q_j \in Q_j,
  \end{alignat}
\end{subequations}
for $j = 1, \dots, A$ and such that $u(\cdot, 0) = u^0(\cdot)$ and
$p_j(\cdot, 0) = p_j^0(\cdot)$ for $j = 1, \dots, A$.

The following lemma is a modified version of Lemma 3.1 in
\cite{Lee2016} and will be used in the energy estimates below.  For
the sake of completeness, we present its proof here.
\begin{lemma}
  \label{lemma:improved:gronwall}
  Let $\mathcal{F}$, $\mathcal{G}$, $\mathcal{G}_1$, $\mathcal{X} :
  [0, T] \rightarrow \R$ be continuous, non-negative
  functions. Suppose that $\mathcal{X}(t)$ satisfies
  \begin{align}
    \label{eq:Q-ineq1}
    \mathcal{X}^2(t) \le C_0 \mathcal{X}^2(0) + C_1 \mathcal{X}(0) + \mathcal{G}_1(t) + \int_0^{t} \left[\mathcal{F}(s) \mathcal{X}(s) + \mathcal{G}(s) \right] \ds,
  \end{align}
  for all $t \in [0, T]$ with constants $C_0 \ge 1$ and $C_1 >
  0$. Then for any $t \in [0, T]$,
  \begin{align}
    \label{eq:diff-ineq}
    \mathcal{X}(t) \lesssim \mathcal{X}(0) + \max \left \{  C_1 + \int_0^{t} \mathcal{F}(s) \ds, \left ( \mathcal{G}_1(t) + \int_0^{t} \mathcal{G}(s) \ds \right )^{\frac{1}{2}} \right \}.
  \end{align}
\end{lemma}
\begin{proof}
  It suffices to show the estimate for the smallest $t$ such that
  \begin{align*}
    \mathcal{X}(t) = \max_{s \in [0,T]} \mathcal{X}(s).  
  \end{align*}
  By this assumption, $\mathcal{X}(t) = \max_{s \in [0,T]}
  \mathcal{X}(s)$ and $\mathcal{X}(s) < \mathcal{X}(t)$ for all $0 \le
  s < t$. We now consider two cases: either
  \begin{align}
    \label{eq:case1}
    C_1 \mathcal{X}(0) + \int_0^{t} \mathcal{F}(s) \mathcal{X}(s)  \ds \ge \mathcal{G}_1(t) + \int_0^{t} \mathcal{G}(s)  \ds
  \end{align}
  or 
  \begin{align}
    \label{eq:case2}
    C_1 \mathcal{X}(0) + \int_0^{t} \mathcal{F}(s) \mathcal{X}(s)  \ds < \mathcal{G}_1(t) + \int_0^{t} \mathcal{G}(s)  \ds .
  \end{align}
  If~\eqref{eq:case1} holds, then~\eqref{eq:Q-ineq1} gives 
  \begin{align*}
    \mathcal{X}^2(t) 
    &\le C_0 \mathcal{X}^2(0) + 2 C_1 \mathcal{X}(0) + 2 \int_0^{t} \mathcal{F}(s) \mathcal{X}(s)  \ds \\
    &\le C_0 \mathcal{X}^2(0) + 2 C_1 \mathcal{X}(0) + 2 \mathcal{X}(t) \int_0^{t} \mathcal{F}(s) \ds .
  \end{align*}
  Dividing both sides by $\mathcal{X}(t)$ yields \eqref{eq:diff-ineq}
  because $\mathcal{X}(t) \ge \mathcal{X}(0)$.
  
  On the other hand, if~\eqref{eq:case2} is the case,
  then~\eqref{eq:Q-ineq1} gives
  \begin{align*}
    \mathcal{X}^2(t) 
    &\le C_0 \mathcal{X}^2(0) + 2 \mathcal{G}_1(t) + 2 \int_0^{t} \mathcal{G}(s)  \ds ,
  \end{align*}
  and taking the square roots of both sides gives
  \eqref{eq:diff-ineq}.
\end{proof}

Theorem~\ref{thm:mpet:qs:ee} below establishes a basic energy estimate
for solutions of~\eqref{eq:mpet:vf}, but also for solutions with an
additional right-hand side (for the sake of reuse in the \emph{a
  priori} error estimates).
\begin{theorem}[Energy estimate for quasi-static multiple-network poroelasticity]
  \label{thm:mpet:qs:ee}
  For given $f \in C^1([0, T]; L^2(\Omega))$, $\beta \in C^1([0, T];
  L^2(\Omega))^{A+1}$ and $\gamma_j \in L^2([0, T]; L^2(\Omega))$ for
  $j = 1, \dots, A$, assume that $u \in C^1([0, T]; V)$ and $p_i \in
  C^1([0, T]; Q_i)$ for $i = 0, \dots, A$ solve
  \begin{subequations}
    \label{eq:ee:quasistatic}
  \begin{alignat}{2}
    \label{eq:ee:quasistatic:1}
    \inner{2 \mu \strain(u)}{\strain(v)} + \inner{p_0}{\Div v} &= \inner{f}{v} &&\quad \foralls v \in V, \\
    \label{eq:ee:quasistatic:2}
    \inner{\Div u}{q_0} - \inner{\lambda^{-1} \alpha \cdot p}{q_0} &= \inner{g_0}{q_0} &&\quad \foralls q_0 \in Q_0, \\
    \label{eq:ee:quasistatic:3}
    \inner{c_j \dot{p}_j + \alpha_j \lambda^{-1} \alpha \cdot \dot{p}  + S_j}{q_j} + \inner{K_j \Grad p_j}{\Grad q_j} &= \inner{g_j}{q_j}
    &&\quad \foralls q_j \in Q_j,
  \end{alignat}
\end{subequations}
  for $j = 1, \dots, A$ and $u(0) = u^0$ and $p_j(0) = p_j^0$ for $j =
  1, \dots, A$, and where $g_0 = - \lambda^{-1} \alpha \cdot \beta$
  and $g_j = \gamma_j + \alpha_j \lambda^{-1} \alpha \cdot
  \dot{\beta}$ for $j = 1, \dots, A$. Then the following energy
  estimate holds for all $t \in (0, T]$:
  \begin{multline}
    \label{eq:thm:ee}
    \| \strain(u(t)) \|_{2 \mu}
    + \sum_{j=1}^A \| p_j(t) \|_{c_j}
    + \| \alpha \cdot p(t) \|_{\lambda^{-1}} \\
    + \left ( \int_{0}^t \sum_{j=1}^A \| \Grad p_j \|_{K_j}^2  
    + \sum_{i, j=1}^A \| p_j - p_i \|_{\xi_{j \leftarrow i}}^2 \ds \right )^{\frac{1}{2}} \\
    \lesssim 
    I_0 
    + \int_{0}^t \left [ \|\dot{f}\| + \| \alpha \cdot \dot{\beta} \|_{\lambda^{-1}} \right ] \ds
    + \left ( \|f(t) \|^2 + \int_{0}^t \sum_{j=1}^A \| \gamma_j \|^2 \ds \right )^{\frac{1}{2}} ,
  \end{multline}
  where
  \begin{equation}
    I_0 = \| \strain(u(0)) \|_{2 \mu}
    + \sum_{j=1}^A \| p_j(0) \|_{c_j}
    + \| \alpha \cdot p(0) \|_{\lambda^{-1}} + \| f(0) \|,
  \end{equation}
  and where the inequality constant is independent of $\lambda$ and
  $c_j$ for $j = 1, \dots, A$, but dependent on $K_j$ for $j = 1,
  \dots, A$.

  Moreover,
  \begin{equation}
    \label{eq:thm:ee:p0}
    \| p_0(t) \|
    \lesssim \| \strain(u(t)) \|_{2 \mu} 
  \end{equation}
  holds.
\end{theorem}
\begin{proof}
  The result follows using standard techniques. Note that the time
  derivative of~\eqref{eq:ee:quasistatic:2} reads as
  \begin{equation} 
    \label{eq:tmp:1}
    \inner{\Div \dot{u}}{q_0} - \inner{\lambda^{-1} \alpha \cdot \dot{p} }{q_0}
    = - \inner{\lambda^{-1} \alpha \cdot \dot{\beta}}{q_0} \quad \foralls q_0 \in Q_0 . 
  \end{equation}
  Taking $v = \dot{u}$ in~\eqref{eq:ee:quasistatic:1}, $q_j = p_j$ for
  $1 \leq j \leq A$ in~\eqref{eq:ee:quasistatic:3} and $q_0 = -p_0$
  in~\eqref{eq:tmp:1}, summing the equations, and rearranging some
  constants (recalling that $\alpha_0 = 1$), we obtain:
  \begin{multline} 
    \label{eq:tmp:2}
      \inner{\strain(u)}{\strain(\dot{u})}_{2 \mu}
      + \sum_{j=1}^A \inner{\dot{p}_j}{p_j}_{c_j}
      + \sum_{j=1}^A \inner{S_j}{p_j}
      + \sum_{j=1}^A \| \Grad p_j \|_{K_j}^2 
      + \inner{\alpha \cdot \dot{p}}{\alpha \cdot p}_{\lambda^{-1}} \\
      = 
      \inner{f}{\dot{u}} + 
      \inner{\lambda^{-1} \alpha \cdot \dot \beta}{\alpha \cdot p} +
      \sum_{j=1}^A \inner{\gamma_j}{p_j}.  
  \end{multline}
  By definition~\eqref{eq:def:transfer}, and the assumption
  that $\xi_{j \leftarrow i} = \xi_{i \leftarrow j}$, it follows that
  \begin{equation}
    \label{eq:tmp:4}
    \sum_{j=1}^A \inner{S_j}{p_j}
    = \sum_{j=1}^A \sum_{i=1}^A \inner{\xi_{j \leftarrow i} (p_j - p_i)}{p_j}
    = \frac{1}{2} \sum_{j=1}^A \sum_{i=1}^A \| p_j - p_i \|_{\xi_{j \leftarrow i}}^2 .
  \end{equation}
  Combining~\eqref{eq:tmp:2} and~\eqref{eq:tmp:4}, and pulling out the
  time derivatives, we find that
  \begin{multline*}
      \frac{1}{2} \ddt \left (
      \| \strain(u) \|_{2 \mu}^2
      + \sum_{j=1}^A \| p_j \|_{c_j}^2
      + \| \alpha \cdot p \|_{\lambda^{-1}}^2 \right )
      + \sum_{j=1}^A \| \Grad p_j \|_{K_j}^2 
      + \frac{1}{2} \sum_{i, j=1}^A \| p_j - p_i \|_{\xi_{j \leftarrow i}}^2 \\
      = 
      \inner{f}{\dot{u}} + 
      \inner{\lambda^{-1} \alpha \cdot \dot \beta}{\alpha \cdot p} +
      \sum_{j=1}^A \inner{\gamma_j}{p_j} . 
  \end{multline*}
  Integrating in time from $0$ to $t$ gives
  \begin{multline}
    \label{eq:thm1:proof:1}
    \| \strain(u(t)) \|_{2 \mu}^2
      + \sum_{j=1}^A \| p_j(t) \|_{c_j}^2
      + \| \alpha \cdot p(t)  \|_{\lambda^{-1}}^2 \\
      + \int_0^t 2 \left[ \sum_{j=1}^A \| \Grad p_j \|_{K_j}^2
      + \sum_{i, j=1}^A \| p_j - p_i \|_{\xi_{j \leftarrow i}}^2 \right ]\ds \\
      = 
      \| \strain(u(0)) \|_{2 \mu}^2
      + \sum_{j=1}^A \| p_j(0) \|_{c_j}^2
      + \| \alpha \cdot p(0)  \|_{\lambda^{-1}}^2 \\
      + 2 \int_0^t \left[ \inner{f}{\dot{u}} 
      + \inner{\lambda^{-1} \alpha \cdot \dot \beta}{\alpha \cdot p} 
      + \sum_{j=1}^A \inner{\gamma_j}{p_j} \right] \ds .
  \end{multline}
  Note first that 
  \begin{align*}
    \int_0^t \inner{f}{\dot{u}} \ds 
    &= \inner{f(t)}{u(t)} - \inner{f(0)}{u(0)} - \int_0^t \inner{\dot{f}}{u} \ds \\
    &\le \| f(t) \| \| u(t) \| + \| f(0) \| \| u(0) \| + \int_0^t \| \dot{f} \| \| u \| \ds \\
    &\lesssim \| f(t) \| \| \strain(u(t)) \|_{2 \mu} + \| f(0) \| \| \strain(u(0)) \|_{2 \mu} + \int_0^t \| \dot{f} \| \| \strain(u) \|_{2 \mu} \ds \\
    &\lesssim \frac{1}{4 \epsilon_0} \| f(t) \|^2 + \epsilon_0 \| \strain(u(t)) \|_{2 \mu}^2 + \| f(0) \| \| \strain(u(0)) \|_{2 \mu} + \int_0^t \| \dot{f} \| \| \strain(u) \|_{2 \mu} \ds ,
  \end{align*}
  using Young's inequality (with $\epsilon$) for any $\epsilon_0 >
  0$. Again using Young's inequality with $\epsilon$, Poincare's
  inequality on $Q_j$ and the assumption of uniform positivity of
  $K_j$ on the last terms on the right hand side
  of~\eqref{eq:thm1:proof:1}, we have that for each $j = 1, \dots, A$
  and any $\epsilon_j > 0$:
  \begin{equation*}
    \inner{\gamma_j}{p_j} 
    \leq \frac{1}{4 \epsilon_j} \| \gamma_j \|^2 + \epsilon_j \| p_j \|^2 
    \lesssim \frac{1}{4 \epsilon_j} \| \gamma_j \|^2 + \epsilon_j \| \Grad p_j \|_{K_j}^2  ,
  \end{equation*}
  with the last inequality depending on $K_j$. Choosing $\epsilon_j$
  for $j = 0, 1, \dots, A$ appropriately and transferring terms thus
  give
  \begin{multline*}
    \| \strain(u(t)) \|_{2 \mu}^2
      + \sum_{j=1}^A \| p_j(t) \|_{c_j}^2
      + \| \alpha \cdot p(t)  \|_{\lambda^{-1}}^2 \\
      + \int_0^t \left[ \sum_{j=1}^A \| \Grad p_j \|_{K_j}^2
      + \sum_{i, j=1}^A \| p_j - p_i \|_{\xi_{j \leftarrow i}}^2 \right] \ds  \\
      \lesssim 
      \| \strain(u(0)) \|_{2 \mu}^2 + \| f(0) \| \| \strain(u(0)) \|_{2 \mu} + \sum_{j=1}^A \| p_j(0) \|_{c_j}^2 + \| \alpha \cdot p(0)  \|_{\lambda^{-1}}^2 + \| f(t) \|^2 \\
      + \int_0^t \sum_{j=1}^A \left[ \| \gamma_j \|^2 +  \| \dot{f} \| \| \strain(u) \|_{2 \mu} + \inner{\lambda^{-1} \alpha \cdot \dot \beta}{\alpha \cdot p} \right] \ds .
  \end{multline*}
  Finally, the Cauchy-Schwarz inequality combined with
  Lemma~\ref{lemma:improved:gronwall}, taking $C_1 = \| f(0) \|$,
  $\mathcal{G}_1(t) = \| f(t) \|^2$, and
  \begin{align*}
    \mathcal{X}(t)^2 &= \| \strain(u) \|_{2 \mu}^2 + \sum_{j=1}^A \| p_j \|_{c_j}^2 + \| \alpha \cdot p \|_{\lambda^{-1}}^2 \\
    & \qquad \qquad + \int_0^t \left[ \sum_{j=1}^A \| \Grad p_j \|^2_{K_j} + \sum_{i, j=1}^A \| p_j - p_i \|^2_{\xi_{j \leftarrow i}} \right] \ds, \\   
    \mathcal{F}(s) &= \| \dot{f}(s) \| + \| \alpha \cdot \dot{\beta}(s) \|_{\lambda^{-1}}, \\
    \mathcal{G}(s) &= \sum_{j=1}^A \| \gamma_j (s) \|^2 ,
  \end{align*}
  give the desired estimate.

  The bound for $p_0$ immediately follows from an inf-sup type
  argument: by the choice of $V$ and $Q_0$, the inf-sup condition (see
  e.g.~\cite{Braess2001}), by~\eqref{eq:mpet:vf:1}, and
  Korn's inequality, we obtain that for any $t \in (0, T]$:
  \begin{equation}
    \| p_0(t) \|
    \lesssim \sup_{v \in V, v \not = 0} \frac{| \inner{\Div v}{p_0(t)} |}{\|v\|_{H^1}}
    = \sup_{v \in V, v \not = 0} \frac{| \inner{2 \mu \strain(u(t))}{\strain(v)}|}{\|v\|_{H^1}} 
    \lesssim \| \strain(u(t)) \|_{2 \mu} 
  \end{equation}
  holds with constant depending on $\mu$.
\end{proof}

We remark that Theorem~\ref{thm:mpet:qs:ee} gives a uniform bound on
$u$ in $L^{\infty}(0, T; V)$, $p_0 \in L^{\infty}(0, T; Q_0)$, and
$p_j$ in $L^2(0, T; Q_j)$ for $j = 1, \dots, A$, for arbitrarily large
$\lambda$ and arbitrarily small $c_j > 0$ for $j = 1, \dots, A$ in
particular.

\section{Semi-discretization of multiple network poroelasticity}
\label{sec:mpet:qs:sd}

In this section, we present a finite element semi-discretization of
the total pressure variational formulation~\eqref{eq:mpet:tp} of the
quasi-static multiple-network poroelasticity equations. We introduce
both abstract compatibility assumptions (\textbf{A1} and \textbf{A2}
below) and a specific choice of conforming, mixed finite element
spaces. We end this section by an \emph{a priori} error estimate for
the discretization error in the abstract case, and an \emph{a priori}
semi-discrete error estimate for a specific family of mixed finite
element spaces.

\subsection{Finite element semi-discretization}
Let $\triang_h$ denote a conforming, shape-regular, simplicial
discretization of $\Omega$ with discretization size $h > 0$. Relative
to $\triang_h$, we define finite element spaces $V_h \subset V$ and
$Q_{i, h} \subset Q_i$ for $i = 0, \dots, A$. We assume that $V_h$ and
$Q_{i, h}$, $i = 0, \dots, A$ satisfy two compatibility assumptions
(\textbf{A1}, \textbf{A2}) as follows:
\begin{itemize}
\item[\bf A1:] $V_h \times Q_{0, h}$ is a stable (in the
  Brezzi~\cite{Brezzi1974} sense) finite element pair for the Stokes
  equations.
\item[\bf A2:] $Q_{j, h}$ is an $H^1$-conforming finite element space
  for $j = 1, \dots, A$.
\end{itemize}
We also denote $Q_h = Q_{0, h} \times Q_{1, h} \times \dots \times
Q_{A, h}$.

With reference to these element spaces, we define the following
semi-discrete total pressure-based variational formulation of the
quasi-static multiple-network poroelasticity equations: for $t \in (0,
T]$, find $u_h(t) \in V_h$ and $p_{i, h}(t) \in Q_{i, h}$ for $i = 0,
  \dots, A$ such that
\begin{subequations}
  \label{eq:quasistatic:fem}
  \begin{alignat}{2}
    \inner{2 \mu \strain(u_h)}{\strain(v)} + \inner{p_{0, h}}{\Div v} &= \inner{f}{v} &&\quad \foralls v \in V_h, \\
    \inner{\Div u_h}{q_0} - \inner{\lambda^{-1} \alpha \cdot p_h}{q_0} &= 0 && \quad \foralls q_0 \in Q_{0, h}, \\
    \inner{c_j \dot{p}_{j, h} +  \alpha_j \lambda^{-1} \alpha \cdot \dot{p}_h + S_{j, h}}{q_j} + \inner{K_j \Grad p_{j, h}}{\Grad q_j} &= \inner{g_j}{q_j}
    &&\quad \foralls q_j \in Q_{j, h},
  \end{alignat}
\end{subequations}
for $j = 1, \dots, A$. Here $S_{j, h} = \sum_{i=1}^A \xi_{j \leftarrow
  i} (p_{j, h} - p_{i, h})$ cf.~\eqref{eq:def:transfer} and $p_h =
(p_{0, h}, \dots, p_{A, h})$.

\subsection{Auxiliary interpolation operators}

As a preliminary step for the \emph{a priori} error analysis of the
semi-discrete formulation, we introduce a set of auxiliary
interpolation operators. In particular, we define interpolation
operators
\begin{align*}
  \Pi_h^V : V \rightarrow V_h, \qquad 
  \Pi_h^{Q_i} : Q_i \rightarrow Q_{i, h} \quad i = 0, \dots, A,
\end{align*}
as follows.

First, for any $(u, p_0) \in V \times Q_0$, we define its interpolant
$(\Pi_h^V u, \Pi_h^{Q_0} p_0) \in V_h \times Q_{0, h}$ as the unique
discrete solution to the Stokes-type system of equations:
\begin{subequations}
  \label{eq:def:stokes:interpolant}
  \begin{alignat}{2}
  \inner{2\mu \strain(\Pi_h^V u)}{\strain(v)} + \inner{\Pi_h^{Q_0} p_0}{\Div v} &= 
  \inner{2\mu \strain(u)}{\strain(v)} + \inner{p_0}{\Div v} && \quad \foralls v \in V_h, \\
  \inner{\Div \Pi_h^V u}{q_0} &= \inner{\Div u}{q_0} && \quad \foralls q_0 \in Q_{0, h}.
\end{alignat}
\end{subequations}
The interpolant is well-defined and bounded by assumption \textbf{A1}
and the given boundary conditions.

Second, for $j = 1, \dots, A$, we define the interpolation operators
$\Pi_h^{Q_j}$ as a weighted elliptic projection: i.e.~for any $p_j \in
Q_{j}$, we define its interpolant $\Pi_h^{Q_j} p_j \in Q_{j, h}$ as
the unique solution of
\begin{equation}
  \label{eq:def:elliptic:interpolant}
  \inner{K_j \Grad \Pi_h^{Q_j} p_j}{q} = \inner{K_j \Grad p_j}{\Grad q}
  \quad \foralls q \in Q_{j, h}.
\end{equation}
This interpolant is well-defined and bounded by assumption \textbf{A2} and the
given boundary conditions.

\subsection{Specific choice of finite element spaces: a family of Taylor-Hood type elements}
\label{sec:taylor-hood}

In this paper, we will pay particular attention to one specific family
of mixed finite element spaces for the total pressure-based
semi-discretization of the multiple-network poroelasticity equations,
namely a family of Taylor-Hood type element spaces
\cite{TaylorHood1973,BercovierPironneau1979}. More precisely, we note
that assumptions \textbf{A1} and \textbf{A2} are easily satisfied by
the conforming mixed finite element space pairing:
\begin{equation}
  \label{eq:def:taylor-hood-type}
  V_h = \mathring{\mathcal{P}}_{l+1}^d(\triang_h), \quad
  Q_{0, h} = \mathcal{P}_{l}(\triang_h), \quad
  Q_{j, h} = \mathring{\mathcal{P}}_{l_j}(\triang_h) ,
\end{equation}
for polynomial degrees $l \geq 1$ and $l_j \geq 1$ for $j = 1, \dots,
A$. We will refer to the spaces~\eqref{eq:def:taylor-hood-type} as
Taylor-Hood type elements of order $l$ and $l_j$. The superimposed
ring in~\eqref{eq:def:taylor-hood-type} denotes the restriction of the
piecewise polynomial spaces to conform to the given essential boundary
conditions.

For this choice of finite element spaces, in particular, for the
Taylor-Hood elements of order $l$, the following error estimate holds
for the Stokes-type interpolant defined
by~\eqref{eq:def:stokes:interpolant} (see
e.g.~\cite{BrezziFalk1991,Boffi1994,Boffi1997}). For $1 \leq m \leq
l+1$, if $u \in H^{m+1}_{0, \Gamma_D}(\Omega)$ and $p_0 \in H^{m}$, then
\begin{equation}
  \label{eq:taylor:hood:ie}
  \| u - \Pi_h^V u \|_{H^1} + \| p_0 - \Pi_h^{Q_0} p_0 \| \lesssim h^m \left (\| u \|_{H^{m+1}} + \| p_0 \|_{H^m} \right ).
\end{equation}

Moreover, the following error estimate holds for the elliptic
interpolants defined by~\eqref{eq:def:elliptic:interpolant} (see
e.g.~\cite[Chap. 5]{BrennerScott2008}): 
For $j = 1, \dots, A$, for $1 \leq m \leq l_j$, if $p_j \in H_{0}^{m+1}$, 
it holds that 
\begin{equation}
  \label{eq:p:ie:H1}
  \| p_j - \Pi_h^{Q_j} p_j \|_{H^1} \lesssim h^m \| p_j \|_{H^{m+1}},
\end{equation}
and under the full elliptic regularity assumption of $\Omega$,
\begin{equation}
  \label{eq:p:ie}
  \| p_j - \Pi_h^{Q_j} p_j \| \lesssim h^{m+1} \| p_j \|_{H^{m+1}}.
\end{equation}
In the next subsection, we show optimal error estimates of
semi-discrete solutions assuming that both of the above estimates
hold.

\subsection{Semi-discrete \emph{a priori} error analysis}

Assume that $(u, p)$ is a solution of the continuous quasi-static
multiple-network poroelasticity equations~\eqref{eq:mpet:vf} and that
$(u_h, p_h)$ solves the corresponding semi-discrete
problem~\eqref{eq:quasistatic:fem}. We introduce the semi-discrete
(approximation) errors
\begin{equation}
  \label{eq:def:errors}
  e_u(t) \equiv u(t) - u_h(t),
  \quad
  e_{p_j}(t) \equiv p_j(t) - p_{j, h}(t) \quad j = 0, \dots, A,
\end{equation}
and denote $e_p = (e_{p_0}, \dots, e_{p_A})$. We also introduce the
standard decomposition of the errors into interpolation (superscript
I) and discretization (superscript $\ea$) errors:
\begin{subequations}
  \label{eq:def:discretization:errors}
  \begin{align}
    e_u &\equiv e_u^I + e_u^{\ea}, \quad e_u^I \equiv u - \Pi_h^V u, \quad e_u^{\ea} \equiv \Pi_h^V u - u_h,  \\
    e_{p_j} &\equiv e_{p_j}^I + e_{p_j}^{\ea}, \quad e_{p_j}^I \equiv p_j - \Pi_h^{Q_j} p_j, \quad e_{p_j}^{\ea} \equiv \Pi_h^{Q_j} p_j - p_{j, h}
    \quad j = 0, \dots, A.  
  \end{align}
\end{subequations}

Proposition~\ref{prop:eh} below provides estimates for the
discretization errors that are robust with respect to $c_j$ and
$\lambda$. In particular, the implicit constants in the estimates are
uniformly bounded for arbitrarily large $\lambda$ and arbitrarily
small $c_j > 0$ for $j = 1, \dots, A$. We also note that the
discretization errors of $u$ in the $L^{\infty}(0, T; V)$-norm and
$p_j$ in the $L^2(0, T; Q_j)$-norms for $j = 1, \dots, A$ converge at
a higher rate than the corresponding interpolation errors, as the
discretization errors are bounded essentially by the initial
discretization error of $u$ in the $V$-norm, by the initial
discretization error of $p_i$ in the $L^2$-norm for $i = 0, \dots, A$
and by the interpolation error of $p_i$ in the $L^2(0, T; L^2)$-norm.

\begin{prop}
  \label{prop:eh}
  Assume that $(u, p) \in C^1(0, T; V) \times C^1(0, T; Q)$ solves the
  total pressure-based variational formulation of the MPET
  equations~\eqref{eq:mpet:vf} for given $f$ and $g_j$ for $j = 1,
  \dots, A$. Assume that $V_h \times Q_h$ satisfies assumptions
  \textbf{A1}-\textbf{A2}, that $(u_h, p_h) \in C^1(0, T; V_h) \times
  C^1(0, T; Q_h)$ solves the corresponding finite element
  semi-discrete problem~\eqref{eq:quasistatic:fem}, and that the
  discretization errors $e_u^{\ea}$ and $e_p^{\ea}$ are defined
  by~\eqref{eq:def:discretization:errors}. Then, the following
  estimate holds for all $t \in (0, T]$:
  \begin{multline}
    \label{eq:prop:eh}
    \| \strain (e_u^{\ea}(t)) \|_{2 \mu} + \sum_{j=1}^A \| e_{p_j}^{\ea}(t) \|_{c_j} + \| \alpha \cdot e_p^{\ea}(t) \|_{\lambda^{-1}} \\
      + \left ( \int_0^t \sum_{j=1}^A \| \Grad e_{p_j}^{\ea} \|_{K_j}^2
      + \sum_{i, j=1}^A \| e_{p_j}^{\ea} - e_{p_i}^{\ea} \|^2_{\xi_{j \leftarrow i}} \ds \right )^{\frac{1}{2}} \\
      \lesssim 
      E_0^h
      + \int_0^t \| \alpha \cdot e_p^I \|_{\lambda^{-1}} \ds +
      \left ( \int_0^t \sum_{j=1}^{A} \| c_j \dot{e}_{p_j}^I + S_j(e_p^I) \|^2 \ds  \right )^{\frac{1}{2}}  ,
  \end{multline}
  with an implicit constant independent of $h$, $T$, $\lambda$, $c_j$
  and $\xi_{j \leftarrow i}$ for $i, j = 1, \dots, A$ where $S_j(e_p)
  = \ssum_{i=1}^A \xi_{j \leftarrow i} (e_{p_j} - e_{p_i})$ and
  \begin{equation}
    \label{eq:def:E0h}
    E_0^h = \| \strain (e_u^{\ea}(0)) \|_{2 \mu} + \ssum_{j=1}^{A} \| e_{p_j}^{\ea}(0) \|_{c_j} + \| \alpha \cdot e_p^{\ea}(0) \|_{\lambda^{-1}}. 
  \end{equation}
  Moreover, for $t \in (0, T]$,
  \begin{equation}
    \label{eq:prop:eh:p0}
    \| e_{p_0}^h(t) \|
    \lesssim \| \strain(e_u^h(t)) \|_{2 \mu} .
  \end{equation}
\end{prop}
\begin{proof}
  A standard subtraction of~\eqref{eq:quasistatic:fem}
  from~\eqref{eq:mpet:vf} gives that the errors $e_u$ and
  $e_p$ satisfy the error equations:
  \begin{subequations}
    \begin{alignat}{2}
      \inner{2 \mu \strain(e_u)}{\strain(v)} + \inner{e_{p_0}}{\Div v} &= 0 &&\quad \foralls v \in V_h, \\
      \inner{\Div e_u}{q_0} - \inner{\lambda^{-1} \alpha \cdot e_p}{q_0} &= 0 && \quad \foralls q_0 \in Q_{0, h}, \\
      \inner{c_j \dot{e}_{p_j} +  \alpha_j \lambda^{-1} \alpha \cdot \dot{e}_p + S_{j}(e_p)}{q_j} + \inner{K_j \Grad e_{p_j}}{\Grad q_j} &= 0
      &&\quad \foralls q_j \in Q_{j, h},
    \end{alignat}
  \end{subequations}
  for $j = 1, \dots, A$ with $S_j(e_p) = \ssum_{i=1}^A \xi_{j
    \leftarrow i} (e_{p_j} - e_{p_i})$.  By the definition of the
  interpolation operators $\Pi_h$, we obtain the reduced error
  representations:
  \begin{subequations}
    \begin{alignat}{2}
      \label{eq:prop:eh:1}
      \inner{2 \mu \strain(e_u^{\ea})}{\strain(v)} + \inner{e_{p_0}^{\ea}}{\Div v} &= 0 &&\quad \foralls v \in V_h, \\
      \inner{\Div e_u^{\ea}}{q_0} - \inner{\lambda^{-1} \alpha \cdot e_p^{\ea}}{q_0} &=
      \inner{g_0^I}{q_0}
      && \quad \foralls q_0 \in Q_{0, h}, \\
      \inner{c_j \dot{e}^{\ea}_{p_j} + \alpha_j \lambda^{-1} \alpha \cdot \dot{e}^{\ea}_p + S_{j}(e_p^{\ea})}{q_j} + \inner{K_j \Grad e_{p_j}^{\ea}}{\Grad q_j}
      &= \inner{g_j^I}{q_j}
      &&\quad \foralls q_j \in Q_{j, h} ,
    \end{alignat}
  \end{subequations}
  for $j = 1, \dots, A$ where $g_0^I = \lambda^{-1} \alpha \cdot
  e_p^I$ and $g_j^I = - c_j \dot{e}^I_{p_j} - \alpha_j \lambda^{-1}
  \alpha \cdot \dot{e}^I_p - S_{j}(e_p^I)$. Noting that $e_u^{\ea}$
  and $e_p^{\ea}$ satisfy the assumptions of
  Theorem~\ref{thm:mpet:qs:ee} with $f = 0$, $\beta = - e_p^I$ and
  $\gamma_j = - c_j \dot{e}^I_{p_j} - S_{j}(e_p^I)$, the semi-discrete
  discretization error estimate~\eqref{eq:prop:eh} follows.

  Further, by the same techniques as used for the
  bound~\eqref{eq:thm:ee:p0}, and assumption \textbf{A1} combined
  with~\eqref{eq:prop:eh:1}, we observe that
  \begin{equation}
    \| e_{p_0}^h(t) \|
    \lesssim \sup_{v \in V_h, v \not = 0} \frac{| \inner{\Div v}{e_{p_0}^h(t)} |}{\|v\|_{H^1}}
    = \sup_{v \in V_h, v \not = 0} \frac{| \inner{2 \mu \strain(e_u^h(t))}{\strain(v)}|}{\|v\|_{H^1}} 
    \lesssim \| \strain(e_u^h(t)) \|_{2 \mu} ,
  \end{equation}
  with constant depending on $\mu$, thus
  yielding~\eqref{eq:prop:eh:p0}.
\end{proof}

We now consider error estimates associated with the specific choice of
Taylor-Hood type finite element spaces as introduced in
Section~\ref{sec:taylor-hood}. Theorem~\ref{thm:mpet:qs:taylor-hood}
below presents a complete semi-discrete error estimate for this case,
and is easily extendable to other elements satisfying \textbf{A1} and
\textbf{A2}.
\begin{theorem}
  \label{thm:mpet:qs:taylor-hood}
  Assume that $(u, p)$ and $(u_h, p_h)$ are defined as in
  Proposition~\ref{prop:eh} over Taylor-Hood type elements of order
  $l$ and $l_j$ for $j = 1, \dots, A$ as defined
  by~\eqref{eq:def:taylor-hood-type}, and that $(e_u, e_p)$ is
  defined by~\eqref{eq:def:errors}. Assume that $(u, p)$ is
  sufficiently regular. Then the following three estimates hold for
  all $t \in (0, T]$ with implicit constants independent of $h$, $T$,
    $\lambda$, $c_j$ and $\xi_{j \leftarrow i}$ for $i, j = 1, \dots,
    A$. First,
  \begin{multline}
    \label{eq:thm:taylor-hood:u}
    \| u(t) - u_h(t) \|_{H^1}
    \lesssim E_0^h + h^{l+1} \left ( \| u(t) \|_{H^{l+2}} + \| u \|_{L^1(0, t; H^{l+2})} + \| p_0 \|_{L^1(0, t; H^{l+1})} \right ) \\
    + \sum_{j=1}^A h^{l_j+1} \left ( \|p_j \|_{L^1(0, t; H^{l_j+1})} + \| \dot{p_j}, p_j \|_{L^2(0, t; H^{l_j+1})} \right ) ,
  \end{multline}
  holds with $E_0^h$ defined in~\eqref{eq:def:E0h}, and 
  \begin{equation*}
  \| \dot{p_j}, p_j \|_{L^2(0, t; H^{l_j +1})} \equiv \| \dot{p}_j \|_{L^2(0, t; H^{l_j+1})} + \| p_j \|_{L^2(0, t; H^{l_j +1})}.
  \end{equation*}
  In addition, 
  \begin{multline}
    \label{eq:thm:taylor-hood:p}
    \sum_{j=1}^A \| p_j - p_{j, h} \|_{L^2(0, t; H^1)} 
    \lesssim E_0^h  
    + h^{l+1} \left ( \| u \|_{L^1(0, t; H^{l+2})} + \| p_0 \|_{L^1(0, t; H^{l+1})}\right ) \\
    + \sum_{j=1}^A  h^{l_j} \| p_j \|_{L^2(0, t; H^{l_j+1})} 
    + h^{l_j+1} \left ( \| p_j \|_{L^1(0, t; H^{l_j+1})} + \| \dot{p_j}, p_j \|_{L^2(0, t; H^{l_j +1})} \right ) 
  \end{multline}
  and 
  \begin{align}
    \label{eq:thm:taylor-hood:p0}
    \| p_0(t) - p_{0, h}(t) \| \lesssim h^{l+1} ( \| p_0(t) \|_{H^{l+1}} + \| u(t) \|_{H^{l+2}}) + \| \strain(e_u^h(t)) \|_{2 \mu}
  \end{align}
  hold.
\end{theorem}
\begin{proof}
  Let $(u, p)$, $(u_h, p_h)$ and $(e_u, e_p)$ be as stated. By the
  triangle inequality, the definition of $e_u^h$, Korn's inequality,
  and \eqref{eq:taylor:hood:ie} for any $t \in (0, T]$, we have that
  \begin{align*}
    \| u(t) - u_h(t) \|_{H^1} &\le \| u(t) - \Pi_h^V u(t) \|_{H^1} + \| \Pi_h^V u(t) -  u_h(t) \|_{H^1} \\
    &\lesssim h^{l+1} \| u(t) \|_{H^{l+2}}  + \| \strain(e_u^h(t)) \|_{2 \mu} ,
  \end{align*}
  with inequality constant depending on $\Omega$ and $\mu$. Further,
  Proposition~\ref{prop:eh} gives for any $t \in (0, T]$ that
  \begin{equation}
     \| \strain(e_u^h(t)) \|_{2 \mu} \lesssim E_0^h
     + \int_0^t \| \alpha \cdot e_p^I \|_{\lambda^{-1}} \ds +
     \left ( \int_0^t \sum_{j=1}^A \| c_j \dot{e}_{p_j}^I + S_j(e_p^I) \|^2 \ds \right )^{\half} ,
  \end{equation}
  where $E_0^h$ is defined
  by~\eqref{eq:def:E0h}. Applying~\eqref{eq:taylor:hood:ie}
  and~\eqref{eq:p:ie}, we note that for any $t \in (0, T]$
  \begin{equation}
    \label{eq:thm:taylor-hood:proof:1}
    \| \alpha \cdot e_p^I (t) \|_{\lambda^{-1}}
    \lesssim h^{l+1} \left (\|u(t)\|_{H^{l+2}} + \| p_0 (t)\|_{H^{l+1}} \right )
    + \sum_{j=1}^A h^{l_j+1}  \| p_j(t) \|_{H^{l_j+1}} .
  \end{equation}
  Similarly, by~\eqref{eq:p:ie} and the definition of $S_j$, we have that 
  \begin{equation}
    \label{eq:thm:taylor-hood:proof:2}
    \sum_{j=1}^A \| c_j \dot{e}_{p_j}^I(t) + S_j(e_p^I(t)) \|
    \lesssim
    \sum_{j=1}^A h^{l_j+1}  \| \dot{p}_j(t) \|_{H^{l_j+1}} + h^{l_j+1} \| p_j(t) \|_{H^{l_j+1}}.
  \end{equation}
  Combining the above estimates and rearranging terms
  yield~\eqref{eq:thm:taylor-hood:u}.

  Turning to the pressures $p_j$, analogously using the triangle
  inequality, \eqref{eq:p:ie:H1}, the Poincar\'e inequality, and the
  assumptions on $K_j$, we have for any $t \in (0, T]$ and any $j = 1,
  \dots, A$ that
  \begin{align*}
    \| p_j - p_{j, h} \|_{L^2(0, t; H^1)} 
    &\leq \| p_j - \Pi_h^{Q_j} p_j \|_{L^2(0, t; H^1)} + \| \Pi_h^{Q_j} p_j - p_{j, h} \|_{L^2(0, t; H^1)} \\
    &\lesssim h^{l_j} \| p_j \|_{L^2(0, t; H^{l_j+1})} +
    \left(  \int_0^t \| \nabla e_{p_j}^h(s) \|_{K_j}^2  \ds \right)^{\half} ,
  \end{align*}
  where the constant in the second inequality depends on $\Omega$ and
  the lower bound on $K_j$.  Using
  Proposition~\ref{prop:eh} together
  with~\eqref{eq:thm:taylor-hood:proof:1}
  and~\eqref{eq:thm:taylor-hood:proof:2}, we thus obtain the estimate
  given by~\eqref{eq:thm:taylor-hood:p}.

  Finally, \eqref{eq:thm:taylor-hood:p0} follows from 
  \begin{align*}
  \| p_0 (t) - p_{0,h}(t) \| \le \| p_0 (t) - \Pi_h^{Q_0} p_0(t) \| + \| \Pi_h^{Q_0} p_0 (t) - p_{0,h}(t) \|, 
  \end{align*}
  \eqref{eq:taylor:hood:ie}, and~\eqref{eq:prop:eh:p0}.
\end{proof}

\begin{rmk}
  \label{rmk:cj0}
  We remark that the estimates of Theorem~\ref{thm:mpet:qs:ee},
  Proposition~\ref{prop:eh}, and Theorem~\ref{thm:mpet:qs:taylor-hood}
  all hold uniformly as $c_j \rightarrow 0$, including in the case
  $c_j = 0$, for any $j = 1, \dots, A$.
\end{rmk}

Theorem~\ref{thm:mpet:qs:taylor-hood} above provides an optimal
estimate for $p_j$ in the $L^{\infty}(0, t; H^1)$-norm for $j = 1,
\dots, A$. Moreover, Proposition~\ref{prop:eh} also yields an optimal
estimate for $p_j$ in the $L^{\infty}(0, t; L^2)$-norm for $j = 1,
\dots, A$, as summarized in Proposition~\ref{prop:L2error:pj} below.
\begin{prop}
  \label{prop:L2error:pj}
  Let $(u, p)$, $(u_h, p_h)$, $(e_u, e_p)$ be as in
  Theorem~\ref{thm:mpet:qs:taylor-hood} and let $c_j > 0$ for $j = 1,
  \dots, A$. Then, the following estimate holds for all $t \in (0, T]$
    with implicit constant independent of $h$, $T$, $\lambda$ and
    $\xi_{j \leftarrow i} \geq 0$ for any $i, j = 1, \dots, A$:
  \begin{multline}
  \label{eq:prop:L2error:pj}
  \sum_{j=1}^A \| e_{p_j} (t) \|
  \lesssim
  E_0^h
  + h^{l+1} \left ( \|u \|_{L^1(0, t; H^{l+2})} + \| p_0 \|_{L^1(0, t; H^{l+1})} \right ) \\
  + \sum_{j=1}^A h^{l_j+1} \left ( \| p_j \|_{H^{l_j+1}} + \| p_j \|_{L^1(0, t; H^{l_j+1})}
  + \| p_j \|_{L^2(0, t; H^{l_j+1})} + \| \dot{p}_j \|_{L^2(0, t; H^{l_j+1})}\right ) 
  \end{multline}
  with $E_0^h$ in~\eqref{eq:def:E0h}.
\end{prop}
\begin{proof}
  Using the triangle inequality and~\eqref{eq:p:ie}, we find that 
  \begin{equation}
  \sum_{j=1}^A \| e_{p_j} \|
  \leq \sum_{j=1}^A \| e_{p_j}^I \|
  +  \| e_{p_j}^h \|
  \lesssim
  \sum_{j=1}^A h^{l_j+1} \| p_j \|_{H^{l_j+1}} + \| e_{p_j}^h \|.
  \end{equation}
  Further, using Proposition~\ref{prop:eh} and the assumption that
  $c_j > 0$ for all $j$, \eqref{eq:thm:taylor-hood:proof:1}, and
  \eqref{eq:thm:taylor-hood:proof:2}, we
  obtain~\eqref{eq:prop:L2error:pj}.
\end{proof}

\section{Numerical convergence experiments}
\label{sec:numerics:convergence}

In this section, we present a set of numerical examples to illustrate
the theoretical results presented. In particular, we examine the
convergence of the numerical approximations for test cases with smooth
solutions. All numerical simulations in this section and in the
subsequent Section~\ref{sec:numerics:brain} were run using the FEniCS
finite element software~\cite{AlnaesEtAl2015} (version 2018.1+), and
the simulation and post-processing code is openly
available~\cite{PiersantiRognes2018}.

\begin{table} [ht]
  \centering
  \begin{tabular}{lcccc}
    \toprule
    $h$ & $\| u (T) - u_h (T)\|$ & Rate & 
    $\| u (T) - u_h (T)\|_{H^1}$ & Rate \\ 
    \midrule
    $H$    & $3.13 \times 10^{-2}$ &      & $7.28 \times 10^{-1}$  &       \\
    $H/2$  & $3.64 \times 10^{-3}$ & 3.11 & $1.98 \times 10^{-1}$ & 1.88  \\
    $H/4$  & $4.35 \times 10^{-4}$ & 3.06 & $5.06 \times 10^{-2}$ & 1.96  \\
    $H/8$  & $5.36 \times 10^{-5}$ & 3.02 & $1.27 \times 10^{-2}$ & 1.99  \\
    $H/16$ & $6.67 \times 10^{-6}$ & 3.01 & $3.19 \times 10^{-3}$ & 2.00  \\
    \midrule
    Optimal & & 3 & &  2 \\
    \bottomrule
    \toprule
    $h$ & 
    $\| p_1 (T) - p_{1,h}(T) \|$ & Rate & 
    $\| p_1 (T) - p_{1,h}(T) \|_{H^1}$ & Rate \\
    \midrule
    $H$    & $3.69 \times 10^{-2}$ &      & $4.21 \times 10^{-1}$ &  \\
    $H/2$  & $9.57 \times 10^{-3}$ & 1.92 & $2.16 \times 10^{-1}$ & 0.96 \\ 
    $H/4$  & $2.47 \times 10^{-3}$ & 1.98 & $1.09 \times 10^{-1}$ & 0.99 \\ 
    $H/8$  & $6.21 \times 10^{-4}$ & 1.99 & $5.45 \times 10^{-2}$ & 1.00 \\ 
    $H/16$ & $1.55 \times 10^{-4}$ & 2.00 & $2.73 \times 10^{-2}$ & 1.00 \\ 
    \midrule
    Optimal & & 2 & & 1 \\
    \bottomrule
    \toprule
    $h$ & 
    $\| p_0 (T) - p_{0,h}(T) \|$ & Rate & 
    \\ 
    \midrule
    $H$    & $1.42 \times 10^{-1}$ &      \\ 
    $H/2$  & $3.10 \times 10^{-2}$ & 2.19 \\ 
    $H/4$  & $7.56 \times 10^{-3}$ & 2.04 \\ 
    $H/8$  & $1.88 \times 10^{-3}$ & 2.01 \\ 
    $H/16$ & $4.70 \times 10^{-4}$ & 2.00 \\ 
    \midrule
    Optimal & & 2  \\
    \bottomrule
  \end{tabular}
  \vspace{1em}
  \caption{Approximation errors and convergence rates for the total
    pressure-based mixed finite element discretization for the smooth
    manufactured test case for a nearly incompressible material
    introduced in Example~\ref{ex:mpet:standard}. We observe that the
    optimal convergence is restored for the total pressure-based
    scheme. This is in contrast to the sub-optimal rates observed with
    the standard scheme (cf.~Table~\ref{tab:mpet:standard}). The
    coarsest mesh size $H$ corresponds to a uniform mesh constructed
    by dividing the unit square into $4 \times 4$ squares and dividing
    each square by a diagonal.}
  \label{tab:ex:tp:1}
\end{table}

\subsection{Convergence in the nearly incompressible case}
\label{sec:num:1}

We consider the manufactured solution test case introduced in
Example~\ref{ex:mpet:standard}. As before, we consider a series of
uniform meshes of the computational domain. The coarsest mesh size $H$
corresponds to a uniform mesh constructed by dividing the unit square
into $4 \times 4$ squares and dividing each square by a diagonal.

We let $V_h \times Q_h$ be the lowest-order Taylor-Hood-type elements,
as defined by~\eqref{eq:def:taylor-hood-type} with $l = 1$ and $l_j =
1$ for $j = 1, \dots, A$, for the semi-discrete total pressure
variational formulation~\eqref{eq:quasistatic:fem}. For this
experiment, we used a Crank-Nicolson discretization in time with time
step size $\Delta t = 0.125$ and $T = 0.5$.  Since the exact solutions
are linear in time, we expected this choice of temporal discretization
to be exact. Indeed, we tested with multiple time step sizes and found
that the errors did not depend on the time step size.

We computed the approximation error of $u_h(T)$ and $p_h(T)$ in the
$L^2$ and $H^1$-norms. The resulting errors for $u_h$, $p_{0, h}$, and
$p_{1, h}$ are presented in Table~\ref{tab:ex:tp:1}, together with
computed convergence rates. The errors and convergence rates of $p_{2,
  h}$ were comparable and analogous to those of $p_{1, h}$ and, for
this reason, not reported here.

From Theorem~\ref{thm:mpet:qs:taylor-hood} and
Proposition~\ref{prop:L2error:pj}, we expect second order convergence
(with decreasing mesh size $h$) for $u(T)$ in the $H^1$-norm, second
order convergence for $p_0(T)$ in the $L^2$-norm, first order
convergence for $p_j(T)$ in the $H^1$-norm and second order
convergence for $p_j(T)$ in the $L^2$-norm (since $c_j > 0$) for $j =
1, \dots, A$. The numerically computed errors are in agreement with
these theoretical results. In particular, we recover the optimal
convergence rates of $2$ for $u_h$ in the $H^1$-norm, $2$ for $p_j$ in
the $L^2$-norm and $1$ for $p_j$ in the $H^1$-norm.

Additionally, we observe that we recover the optimal convergence rate
of $3$ for $u_h(T)$ in the $L^2$-norm for this test case. Further
investigations indicate that this does not hold for general $\nu$:
with $\nu = 0.4$, the convergence rate for $u_h(T)$ in the $L^2$-norm
is reduced to between $2$ and $3$, cf.~Table \ref{tab:ex:tp:nu0.4}.

\begin{table} [ht]
  \centering
  \begin{tabular}{lcccc}
    \toprule
    $h$ & $\| u (T) - u_h (T)\|$ & Rate & 
    $\| u (T) - u_h (T)\|_{H^1}$ & Rate \\ 
    \midrule
    $H$    & $3.12 \times 10^{-2}$ &      & $7.25 \times 10^{0}$  &       \\
    $H/2$  & $3.86 \times 10^{-3}$ & 3.02 & $1.98 \times 10^{-1}$ & 1.87  \\
    $H/4$  & $5.47 \times 10^{-4}$ & 2.82 & $5.08 \times 10^{-2}$ & 1.96  \\
    $H/8$  & $9.90 \times 10^{-5}$ & 2.47 & $1.28 \times 10^{-2}$ & 1.99  \\
    $H/16$ & $2.19 \times 10^{-5}$ & 2.18 & $3.20 \times 10^{-3}$ & 2.00  \\
    \bottomrule
  \end{tabular}
  \vspace{1em}
  \caption{Displacement approximation errors and convergence rates for
    the total pressure-based mixed finite element discretization for
    the smooth manufactured test case introduced in
    Example~\ref{ex:mpet:standard} but with $\nu = 0.4$. The coarsest
    mesh size $H$ corresponds to a uniform mesh constructed by
    dividing the unit square into $4 \times 4$ squares and dividing
    each square by a diagonal. We note that the third order
    convergence rate for $u_h(T)$ in the $L^2$-norm observed in
    Table~\ref{tab:ex:tp:nu0.4} is reduced to order $2-3$ in this case
    with $\nu = 0.4$.}
  \label{tab:ex:tp:nu0.4}
\end{table}

\subsection{Convergence in the vanishing storage coefficient case}
We also considered the same test case, total-pressure-based
discretization, and set-up as described in Section~\ref{sec:num:1},
but now with $c_j = 0$ for $j = 1, 2$. The corresponding errors are
presented in Table~\ref{tab:ex:tp:2}. We note that we observe the same
optimal convergence rates as before for this case with $c_j = 0$.
\begin{table}
  \centering
  \begin{tabular}{lcccc}
    \toprule
    $h$ & $\| u (T) - u_h (T)\|$ & Rate & 
    $\| u (T) - u_h (T)\|_{H^1}$ & Rate \\ 
    \midrule
    $H$    & $3.13 \times 10^{-2}$ &      & $7.28 \times 10^{-1}$  &       \\
    $H/2$  & $3.64 \times 10^{-3}$ & 3.11 & $1.98 \times 10^{-1}$ & 1.88  \\
    $H/4$  & $4.35 \times 10^{-4}$ & 3.06 & $5.06 \times 10^{-2}$ & 1.96  \\
    $H/8$  & $5.36 \times 10^{-5}$ & 3.02 & $1.27 \times 10^{-2}$ & 1.99  \\
    $H/16$ & $6.67 \times 10^{-6}$ & 3.01 & $3.19 \times 10^{-3}$ & 2.00  \\
    \midrule
    Optimal & & 3 & &  2 \\
    \bottomrule
    \toprule
    $h$ & 
    $\| p_1 (T) - p_{1,h}(T) \|$ & Rate & 
    $\| p_1 (T) - p_{1,h}(T) \|_{H^1}$ & Rate \\
    \midrule
    $H$    & $3.95 \times 10^{-2}$ &      & $4.21 \times 10^{-1}$ &  \\
    $H/2$  & $1.06 \times 10^{-2}$ & 1.90 & $2.16 \times 10^{-1}$ & 0.96 \\ 
    $H/4$  & $2.69 \times 10^{-3}$ & 1.97 & $1.09 \times 10^{-1}$ & 0.99 \\ 
    $H/8$  & $6.75 \times 10^{-4}$ & 1.99 & $5.45 \times 10^{-2}$ & 1.00 \\ 
    $H/16$ & $1.69 \times 10^{-4}$ & 2.00 & $2.73 \times 10^{-2}$ & 1.00 \\ 
    \midrule
    Optimal & & 2 & & 1 \\
    \bottomrule
    \toprule
    $h$ & 
    $\| p_0 (T) - p_{0,h}(T) \|$ & Rate & & \\
    \midrule
    $H$    & $1.46 \times 10^{-1}$ &      & & \\
    $H/2$  & $3.25 \times 10^{-2}$ & 2.17 & & \\
    $H/4$  & $7.97 \times 10^{-3}$ & 2.03 & & \\
    $H/8$  & $1.99 \times 10^{-3}$ & 2.00 & & \\
    $H/16$ & $4.96 \times 10^{-4}$ & 2.00 & & \\
    \midrule
    Optimal & & 2 \\
    \bottomrule
  \end{tabular}
  \vspace{1em}
  \caption{Approximation errors and convergence rates for the total
    pressure-based mixed finite element discretization for the smooth
    manufactured test case introduced in
    Example~\ref{ex:mpet:standard} but with vanishing storage
    coefficients ($c_j = 0$ for $j = 1, 2$). We observe the optimal
    convergence also for this set of parameter values. The coarsest
    mesh size $H$ corresponds to a uniform mesh constructed by
    dividing the unit square into $4 \times 4$ squares and dividing
    each square by a diagonal. }
  \label{tab:ex:tp:2}
\end{table}

\subsection{Convergence of the discretization error}

Proposition~\ref{prop:eh} indicates superconvergence of the
discretization errors $e_u^h$ and $e_{p_j}^h$. In particular, this
result predicts that for the lowest-order Taylor-Hood-type elements,
we expect to observe second order convergence for the discretization
error of $p_j$ in the $L^2(0, T; H^1)$-norm. To examine this
numerically, we consider the same test case, total-pressure-based
discretization, and set-up as described in Section~\ref{sec:num:1},
but now compute the error between the elliptic interpolants and the
finite element approximation. The results are given in
Table~\ref{tab:ex:tp:3} for $p_1$. The numerical results were
entirely analogous for $p_2$ and therefore not shown. We indeed
observe the second order convergence of $e_{p_j}^h(T)$ (for $j = 1,
2$) in the $H^1$-norm as indicated by~Proposition~\ref{prop:eh}.
\begin{table}
  \centering
  \begin{tabular}{lcccc}
    \toprule
    $h$ & 
    $\| \Pi_h^1 p_1 (T) - p_{1,h}(T) \|$ & Rate & 
    $\| \Pi_h^1 p_1 (T) - p_{1,h}(T) \|_{H^1}$ & Rate \\
    \midrule
    $H$    & $2.98 \times 10^{-3}$ &      & $1.46 \times 10^{-2}$ &      \\
    $H/2$  & $9.12 \times 10^{-4}$ & 1.71 & $4.25 \times 10^{-2}$ & 1.78 \\ 
    $H/4$  & $2.40 \times 10^{-4}$ & 1.92 & $1.11 \times 10^{-2}$ & 1.94 \\ 
    $H/8$  & $6.09 \times 10^{-5}$ & 1.98 & $2.79 \times 10^{-2}$ & 1.99 \\ 
    $H/16$ & $1.53 \times 10^{-5}$ & 2.00 & $6.99 \times 10^{-2}$ & 2.00 \\ 
    \midrule
    Theoretical & & 2 & & 2 \\
    \bottomrule
  \end{tabular}
  \vspace{1em}
  \caption{Discretization errors and convergence rates for $p_1$ for
    the total pressure-based mixed finite element discretization for
    the smooth manufactured test case for a nearly incompressible
    material introduced in Example~\ref{ex:mpet:standard}. We indeed
    observe the higher (second) order convergence of $e_{p_1}^h(T)$ in
      the $H^1$-norm as indicated by~Proposition~\ref{prop:eh}. The
      coarsest mesh size $H$ corresponds to a uniform mesh constructed
      by dividing the unit square into $4 \times 4$ squares and
      dividing each square by a diagonal. }
  \label{tab:ex:tp:3}
\end{table}

\section{Simulating fluid flow and displacement in a human brain using a 4-network model}
\label{sec:numerics:brain}

In this section, we consider a variant of the 4-network model
presented in \cite{TullyVentikos2011} defined over a human brain mesh
with physiologically inspired parameters and boundary conditions. In
particular, we consider the MPET equations~\eqref{eq:mpet} with $A =
4$. The original 4 networks of \cite{TullyVentikos2011} represent (1)
interstitial fluid-filled extracellular spaces, (2) arteries, (3)
veins and (4) capillaries. In view of recent
findings~\cite{AbbottEtAl2018} however, we conjecture that it may be
more physiologically interesting to interpret the extracellular
compartment as a paravascular network.

The computational domain is defined by Version 2 of the Colin 27 Adult
Brain Atlas FEM mesh~\cite{Fang2010}, in particular a coarsened
version of this mesh with $99\,605$ cells and $29\,037$ vertices, and
is illustrated in Figure~\ref{fig:brain:mesh} (left). The domain
boundary consists of the outer surface of the brain, referred to below
as the \emph{skull}, and of inner convexities, referred to as the
\emph{ventricles}, cf.~Figure~\ref{fig:brain:mesh} (right). We
selected three points in the domain $x_0 = (89.9, 108.9, 82.3)$
(center), $x_1 = (102.2, 139.3, 82.3)$ (point in the central z-plane),
and $x_2 = (110.7, 108.9, 98.5)$ (point in the central y-plane). The
relative locations of these points within the domain are also
illustrated in Figure~\ref{fig:brain:mesh} (left).
 \begin{center}
   \begin{figure}
     \includegraphics[width=0.48\textwidth]{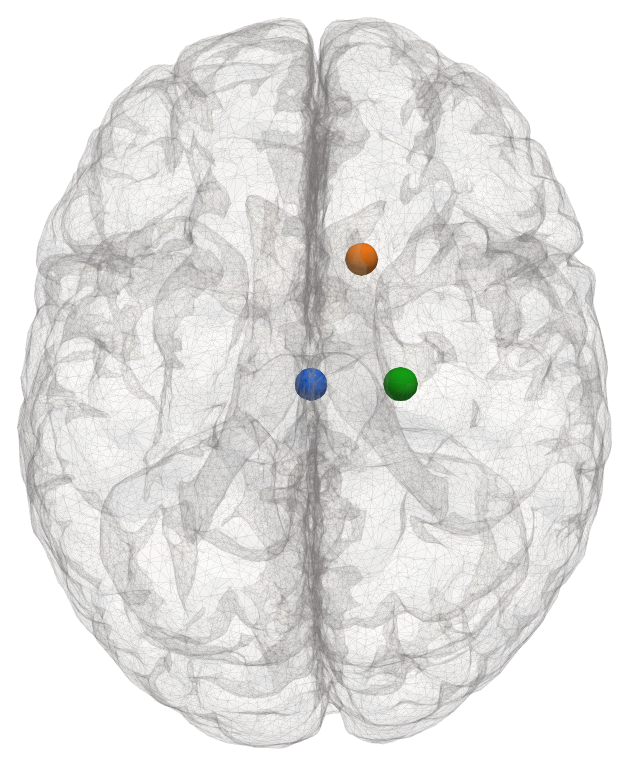}
     \includegraphics[width=0.48\textwidth]{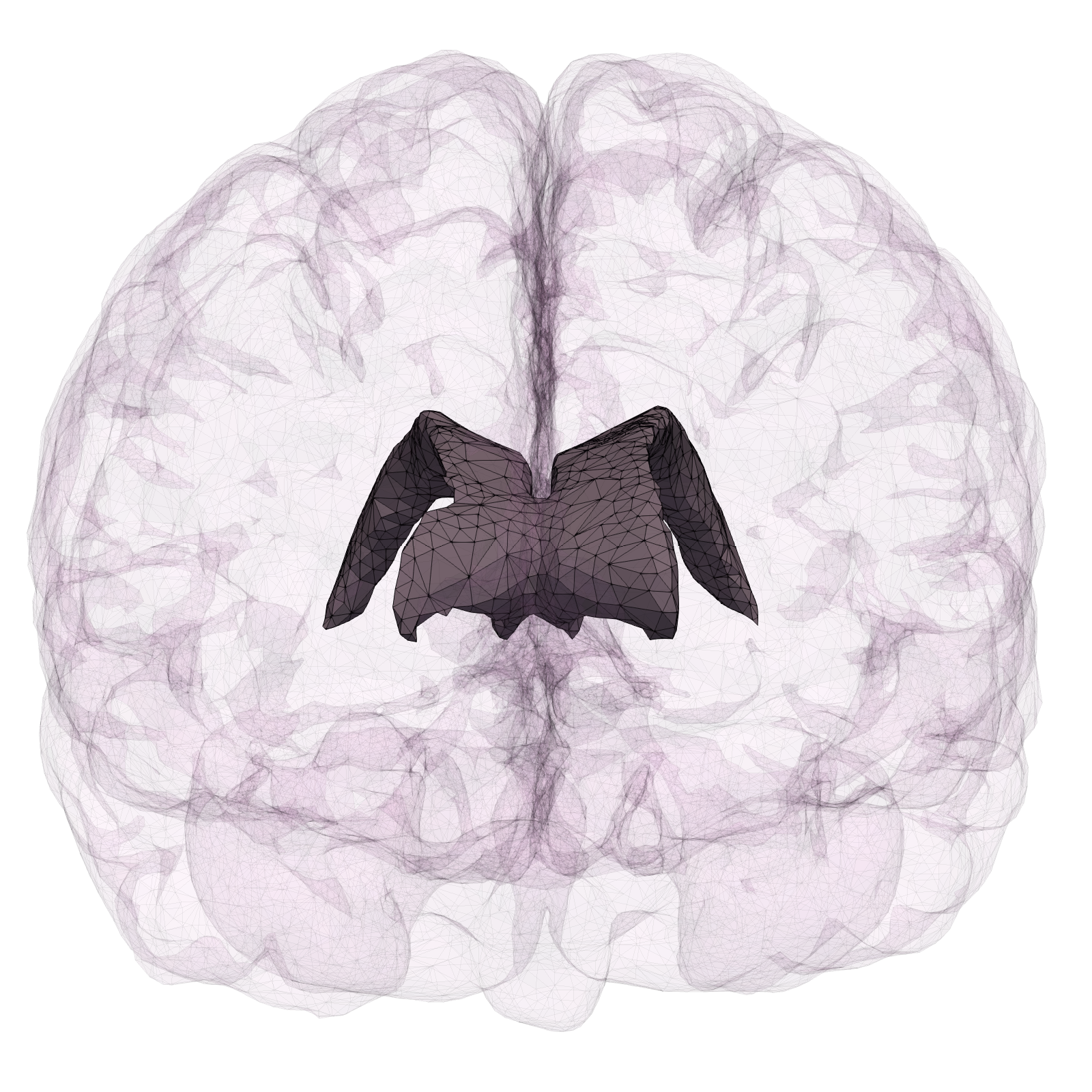}
     \caption{Left: The human brain computational mesh used in
       Section~\ref{sec:numerics:brain} with $99\,605$ cells and
       $29\,037$ vertices. View from top i.e.~along the negative
       z-axis. The points $x_0$ (blue), $x_1$ (orange), $x_2$ (green)
       are marked with spheres.  Right: The inner (ventricular)
       boundaries of the computational mesh. View from front
       i.e.~along negative y-axis.}
     \label{fig:brain:mesh}
   \end{figure}
 \end{center}

 \begin{table} [h]
\begin{center}
  \begin{tabular}{cccc}
    \toprule
    Symbol & Value(s) & Units & Reference \\
    \midrule
    $\nu$ & $0.4999$ & & Comparable with~\cite{NagashimaEtAl1987} \\
    $E$ & $1500$ & Pa & Comparable with~\cite{BuddayEtAl2015} \\
    $c_1$ & $3.9 \times 10^{-4}$ & Pa$^{-1}$ & \cite[Table 2]{GuoEtAl2018} \\
    $c_2, c_4$ & $2.9 \times 10^{-4}$ & Pa$^{-1}$ & \cite[Table 2]{GuoEtAl2018} \\
    $c_3$ & $1.5 \times 10^{-5}$ & Pa$^{-1}$ & \cite[Table 2]{GuoEtAl2018} \\
    $\alpha_1$ & 0.49 & & \cite[Table 2]{GuoEtAl2018} \\
    $\alpha_2, \alpha_4$ & 0.25 & & \cite[Table 2]{GuoEtAl2018} \\
    $\alpha_3$ & 0.01 & & \cite[Table 2]{GuoEtAl2018} \\
    $K_1$ & $1.57\cdot 10^{-5}$ & mm$^{2}$ Pa$^{-1}$ s$^{-1}$ & \cite[Table 1]{VardakisEtAl2016} \\
    $K_2, K_3, K_4, $ & $3.75\cdot 10^{-2}$ & mm$^{2}$ Pa$^{-1}$ s$^{-1}$ & \cite[Table 1]{VardakisEtAl2016} \\
    $\xi_{2 \leftarrow 4}, \xi_{4 \leftarrow 3}, \xi_{4 \leftarrow 1}, \xi_{1 \leftarrow 3}$ & $1.0 \times 10^{-6}$ & Pa$^{-1}$ s$^{-1}$ &
    Comparable with~\cite{MichlerEtAl2013}\\
    $\xi_{1 \leftarrow 2}, \xi_{2 \leftarrow 3}$ & $0.0$ & Pa$^{-1}$ s$^{-1}$ &
    \cite{VardakisEtAl2016}\\
    \bottomrule
  \end{tabular}
  \vspace{1em}
  \caption{Material parameters used for the multiple network
    poroelasticity equations~\eqref{eq:mpet} with $A = 4$ networks for
    the numerical experiments in Section~\ref{sec:numerics:brain}. We
    remark that a wide range of parameter values can be found in the
    literature and the ones used here represents one sample set of
    representative values.}
  \label{tab:brain-params}
\end{center}
\end{table}
We consider the following set of boundary conditions for the system
for all $t \in (0, T)$. All boundary pressure values are given in mmHg
below, noting that $1$ mmHg $\approx 133.32$ Pa. We assume that the
displacement is fixed on the outer boundary and prescribe a total
stress on the inner boundary:
\begin{equation*}
    u = 0 \quad \text{on skull}, \quad 
    (C \varepsilon(u) - \ssum_{j=1}^4 \alpha_j p_j I ) \cdot n = s \, n \quad \text{on ventricles},
\end{equation*}
where $n$ is the outward boundary normal and $s$ is defined as
\begin{equation*}
  s = - \sum_{j=1}^4 \alpha_j \tilde p_j,
\end{equation*}
where $\tilde p_j$ for $j = 1, \dots, 4$ are given below. We assume
that the fluid in network 1 is in direct communication with the
surrounding cerebrospinal fluid, and that a cerebrospinal fluid
pressure is prescribed. In particular, we assume that the
cerebrospinal fluid pressure pulsates around a baseline pressure of
$5$ (mmHg) with a peak transmantle pressure difference magnitude of
$\delta = 0.012$ (mmHg):
\begin{equation*}
p_1 = 5 + 2 \sin(2 \pi t) \quad \text{on skull}, \quad
p_1 = 5 + (2 + \delta) \sin(2 \pi t) \equiv \tilde{p}_1 \quad \text{on ventricles}. 
\end{equation*}
We assume that a pulsating arterial blood pressure is prescribed at
the outer boundary, while on the inner boundaries, we assume no
arterial flux:
\begin{equation*}
  p_2 = 70 + 10 \sin(2 \pi t) \equiv \tilde{p}_2 \quad \text{on skull}, \quad
  \Grad p_2 \cdot n = 0 \quad \text{on ventricles}.
\end{equation*}
For the venous compartment, we assume that a constant pressure is
prescribed at both boundaries:
\begin{equation*}
  p_3 = 6 \equiv \tilde{p}_3 \quad \text{on skull and ventricles}.
\end{equation*}
Finally, for the capillary compartment, we assume no flux at both
boundaries:
\begin{equation*}
  \Grad p_4 \cdot n = 0 \quad \text{on skull and ventricles}.
\end{equation*}

We consider the following initial conditions:
\begin{equation*}
  u = 0, \quad p_1 = 5, \quad p_2 = 70,
  \quad p_3 = 6, \quad p_4 = (p_2 + p_3)/2 \equiv \tilde{p}_4 ,
\end{equation*}
and material parameters as reported in Table~\ref{tab:brain-params}.

 \begin{center}
  \begin{figure}
   \centering
   \begin{subfigure}[b]{0.49\textwidth}
     \centering
     \includegraphics[width=\textwidth]{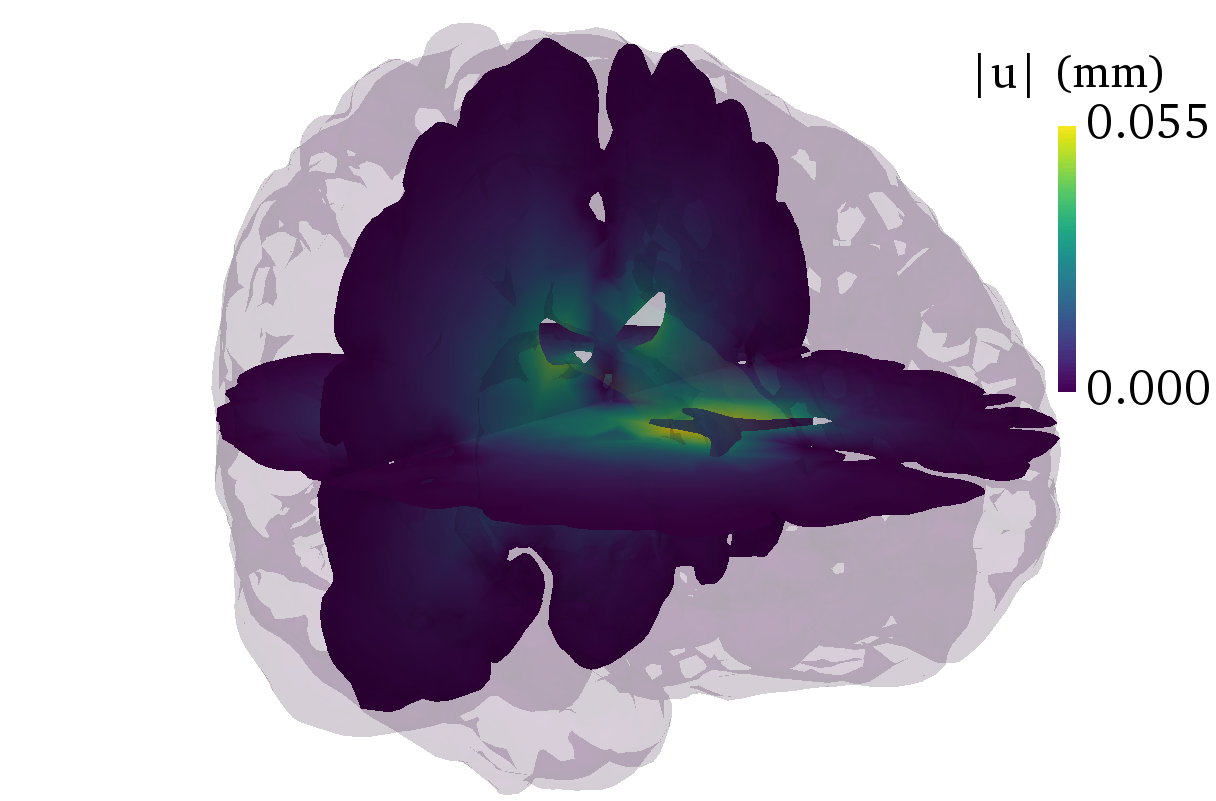} 
     \caption{Displacement magnitude $|u(\bar{t})|$}
   \end{subfigure}
   \begin{subfigure}[b]{0.49\textwidth}
     \centering
     \includegraphics[width=\textwidth]{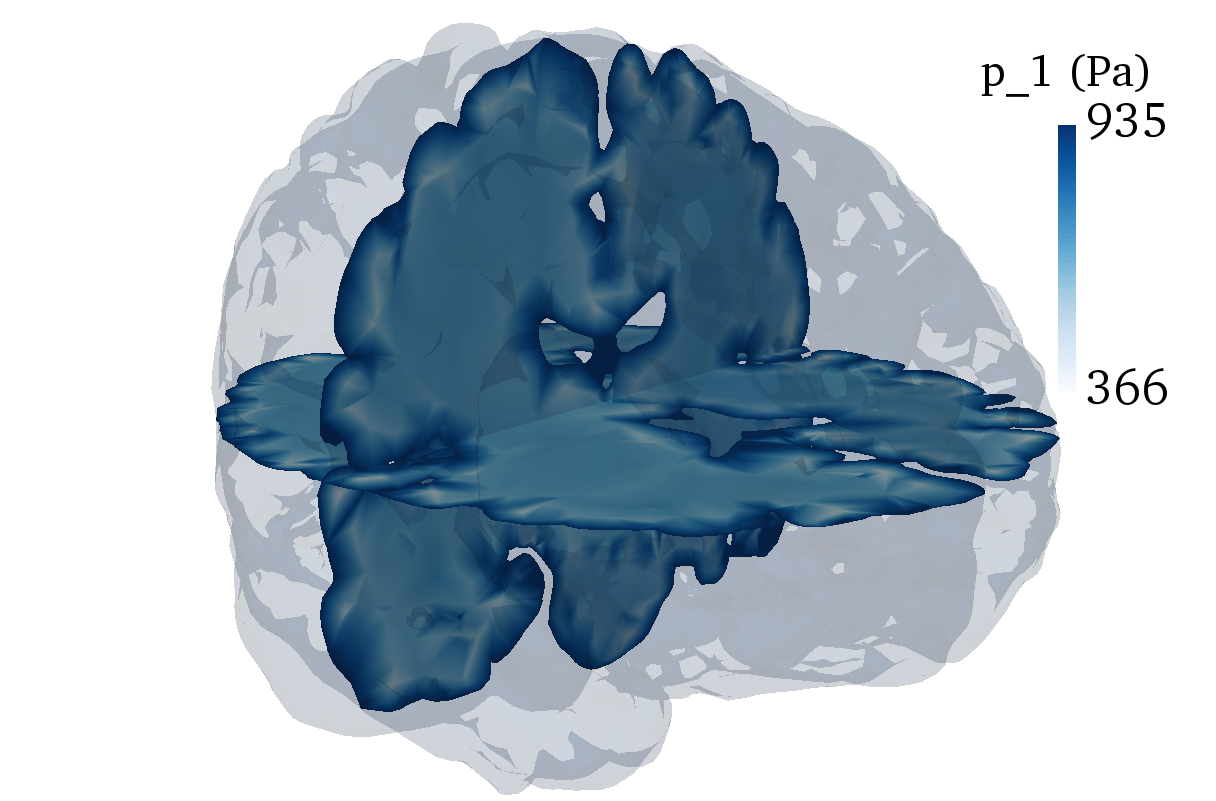} 
     \caption{Extracellular pressure $p_1(\bar{t})$}
   \end{subfigure}
   \vspace{-1em}
   \begin{subfigure}[b]{0.49\textwidth}
     \centering
     \includegraphics[width=\textwidth]{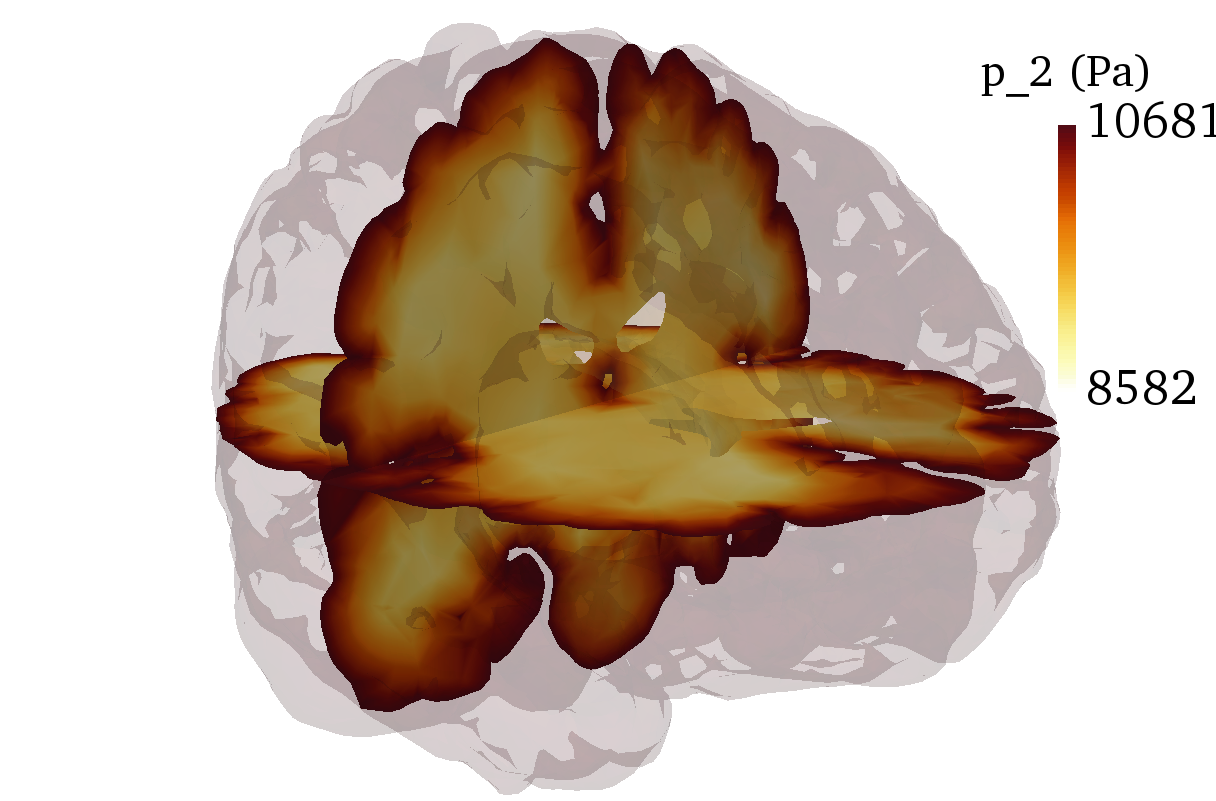} 
     \caption{Arterial pressure $p_2(\bar{t})$}
   \end{subfigure}
   \begin{subfigure}[b]{0.49\textwidth}
     \centering
     \includegraphics[width=\textwidth]{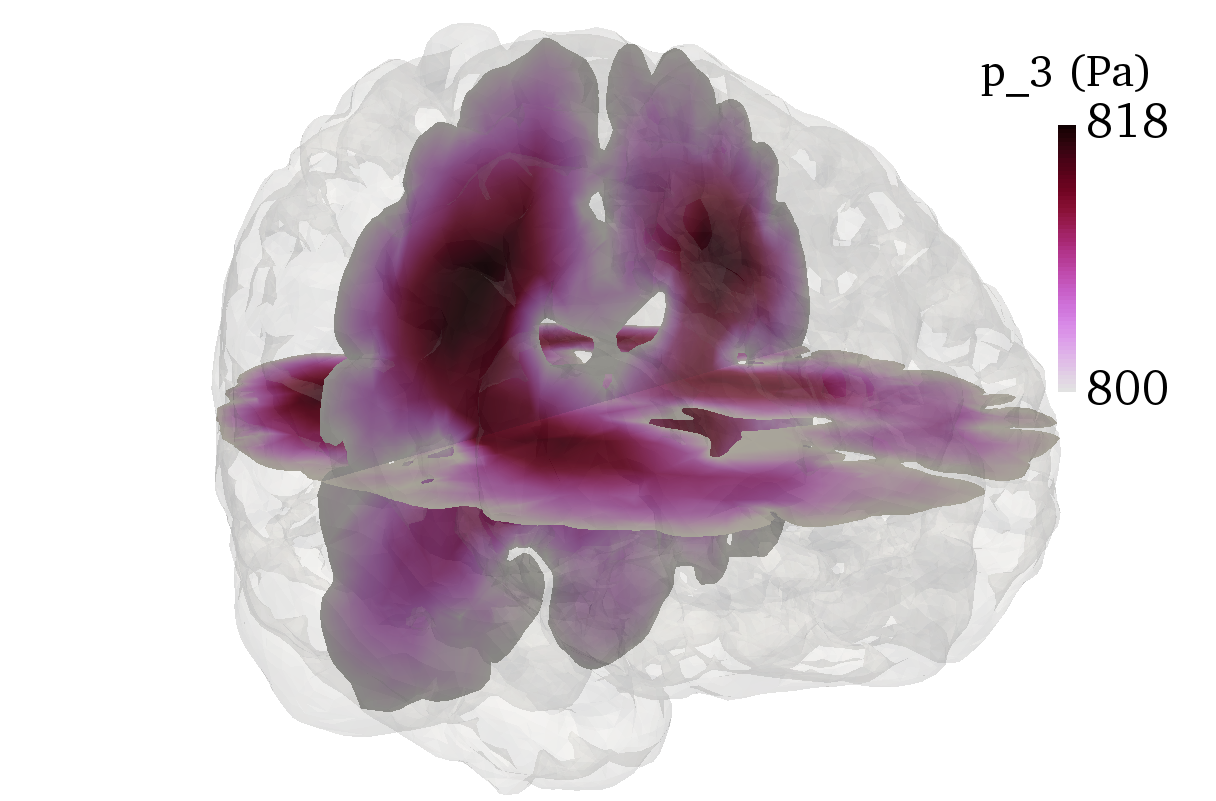} 
     \caption{Venous pressure $p_3(\bar{t})$}
   \end{subfigure}
   \vspace{-1em}
   \begin{subfigure}[b]{0.49\textwidth}
     \centering
     \includegraphics[width=\textwidth]{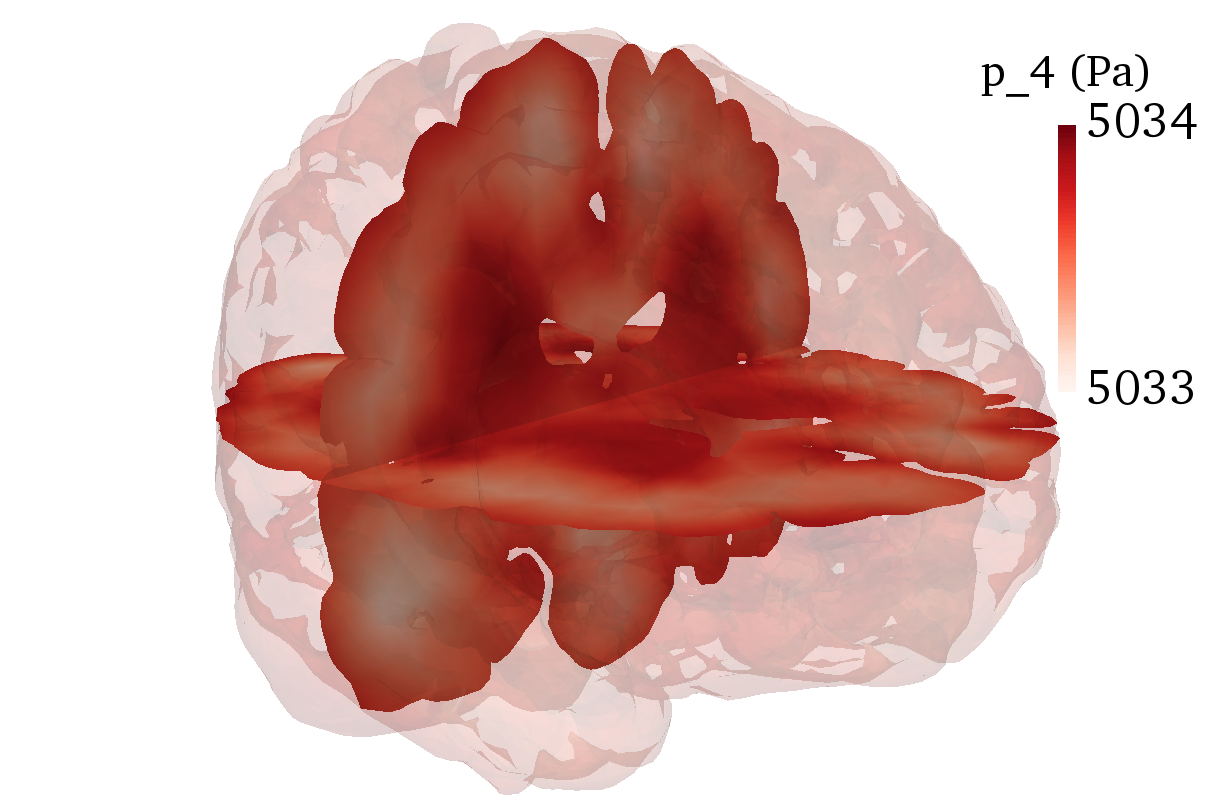} 
     \caption{Capillary pressure $p_4(\bar{t})$}
   \end{subfigure}
   \caption{Results of numerical experiment described in
     Section~\ref{sec:numerics:brain} using the total pressure
     formulation. Plots show slices of computed quantities at $\bar{t}
     = 2.25$ (s) corresponding to the peak arterial inflow in the 2nd
     cycle. From left to right and top to bottom: (a) displacement
     magnitude $|u|$, (b) extracellular pressure $p_1$, (c) arterial
     blood pressure $p_2$, (d) venous blood pressure $p_3$ and (e)
     capillary blood pressure $p_4$.}
  \label{fig:snaps}
  \end{figure}
 \end{center}
 \begin{center}
 \begin{figure}
   \centering
   \begin{subfigure}[b]{0.49\textwidth}
     \centering
     \includegraphics[width=\textwidth]{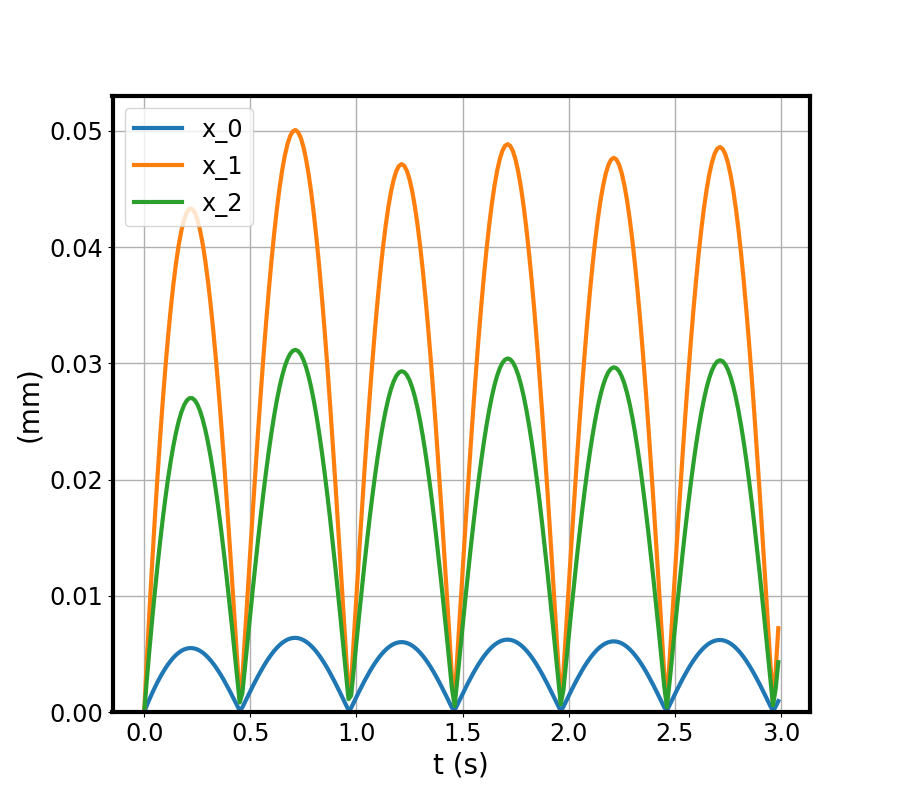} 
     \caption{Displacement magnitude $|u(x_i)|$}
   \end{subfigure}
   \begin{subfigure}[b]{0.49\textwidth}
     \centering
     \includegraphics[width=\textwidth]{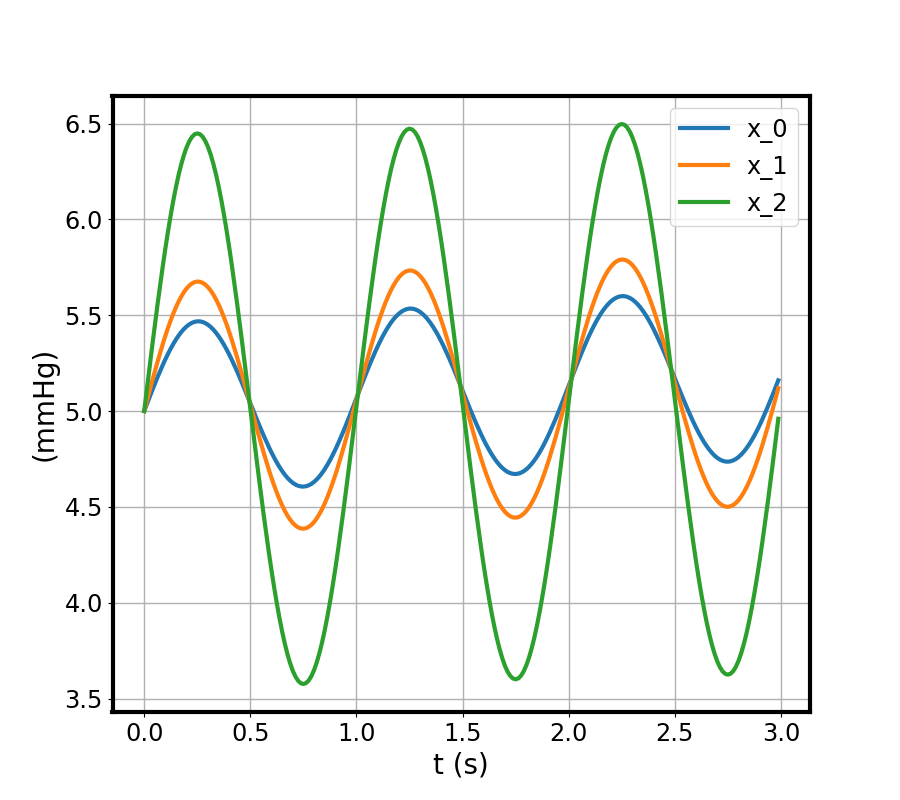} 
     \caption{Extracellular pressure $p_1(x_i)$}
   \end{subfigure}
   \vspace{-1em}
   \begin{subfigure}[b]{0.49\textwidth}
     \centering
     \includegraphics[width=\textwidth]{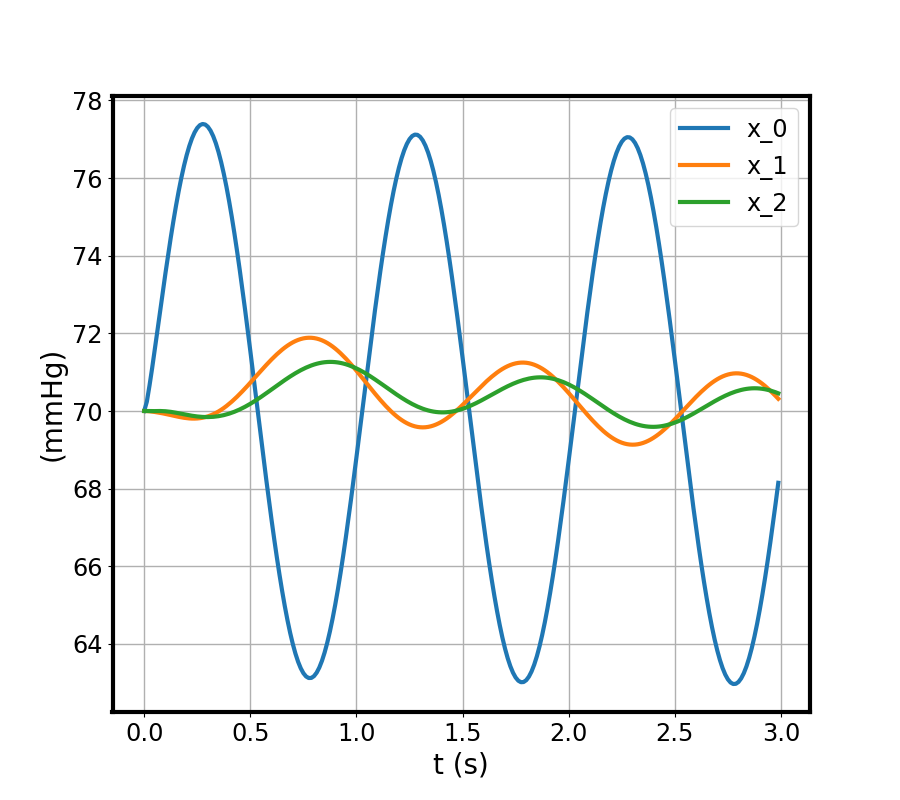} 
     \caption{Arterial pressure $p_2(x_i)$}
   \end{subfigure}
   \begin{subfigure}[b]{0.49\textwidth}
     \centering
     \includegraphics[width=\textwidth]{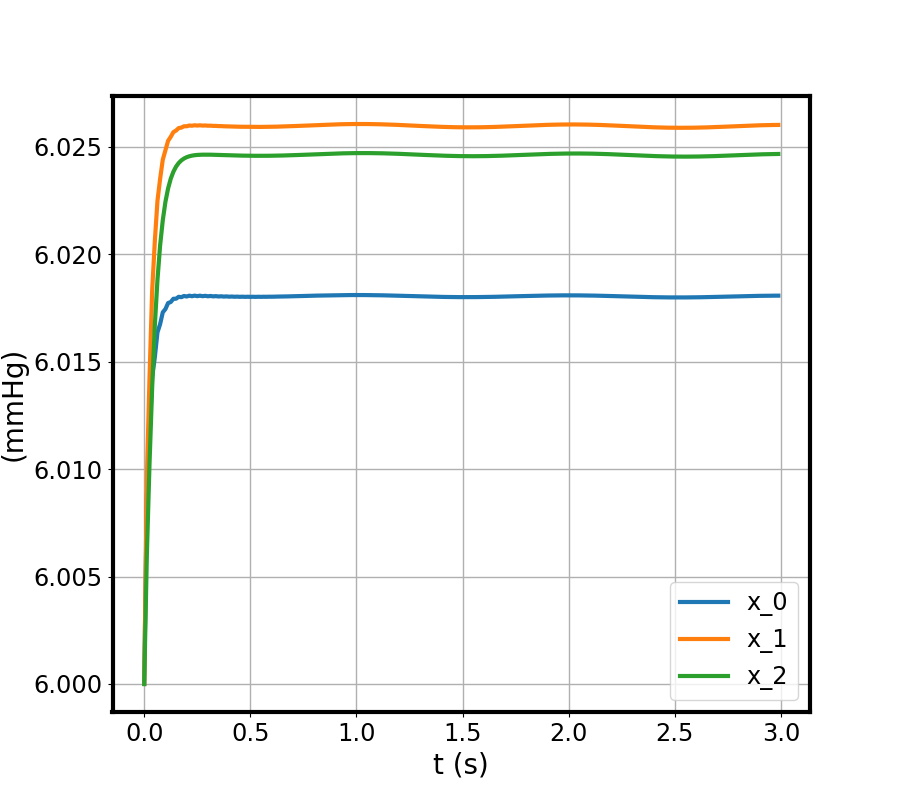} 
     \caption{Venous pressure $p_3(x_i)$}
   \end{subfigure}
   \vspace{-1em}
   \begin{subfigure}[b]{0.49\textwidth}
     \centering
     \includegraphics[width=\textwidth]{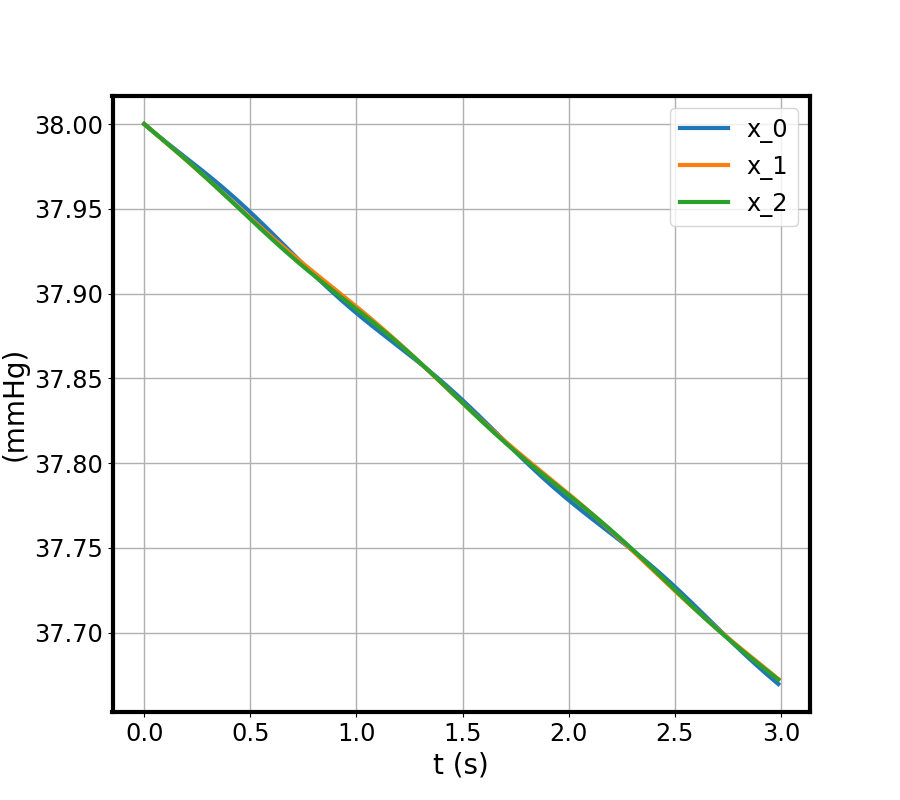} 
     \caption{Capillary pressure $p_4(x_i)$}
   \end{subfigure}
   \caption{Results of numerical experiment described in
     Section~\ref{sec:numerics:brain} using the total pressure
     formulation. Plots show computed quantities over time $t \in (0.0,
     3.0)$ for a set of three points $x_0$, $x_1$, $x_2$. See
     Figure~\ref{fig:brain:mesh} for the location and precise
     coordinates of the points $x_i$. From left to right and top to
     bottom: (a) displacement magnitude $|u|$, (b) extracellular
     pressure $p_1$, (c) arterial blood pressure $p_2$, (d) venous
     blood pressure $p_3$ and (e) capillary blood pressure $p_4$.}
   \label{fig:overtime}
 \end{figure}
 \end{center}

We computed the resulting solutions using the total pressure mixed
finite element formulation with the lowest order Taylor-Hood type
elements ($l = 1$ and $l_j = 1$ for $j = 1, \dots, 4$
in~\eqref{eq:def:taylor-hood-type}), a Crank-Nicolson type
discretization in time with time step $\Delta t = 0.0125$ (s) over the
time interval $(0.0, 3.0)$ (s). The linear systems of equations were
solved using a direct solver (MUMPS). For comparison, we also computed
solutions with a standard mixed finite element formulation (as
described and used in Example~\ref{ex:mpet:standard}) and otherwise
the same numerical set-up.

The numerical results using the total pressure formulation are
presented in Figures~\ref{fig:snaps} and~\ref{fig:overtime}. In
particular, snapshots of the displacement and network pressures at
peak arterial inflow in the 3rd cycle ($t = 2.25$ (s)) are presented
in Figure~\ref{fig:snaps}. Plots of the displacement magnitude and
network pressures in a set of points versus time are presented in
Figure~\ref{fig:overtime}. 

We also compared the solutions computed using the total pressure and
standard mixed finite element formulation. Plots of the displacement
magnitude in a set of points over time are presented in
Figure~\ref{fig:brain:comp}. We clearly observe that the computed
displacements using the two formulations differ. For instance, the
displacement magnitude in the point $x_0$ computed using the standard
formulation is less than half the magnitude computed using the total
pressure formulation. We also visually compared the pressures computed
using the two formulations and found only minimal differences for this
test case (data not shown for the standard formulation).
 \begin{center}
   \begin{figure}
   \begin{subfigure}[b]{0.49\textwidth}
     \centering
     \includegraphics[width=\textwidth]{u_mag.png} 
     \caption{Total pressure formulation}
   \end{subfigure}
   \begin{subfigure}[b]{0.49\textwidth}
     \centering
     \includegraphics[width=\textwidth]{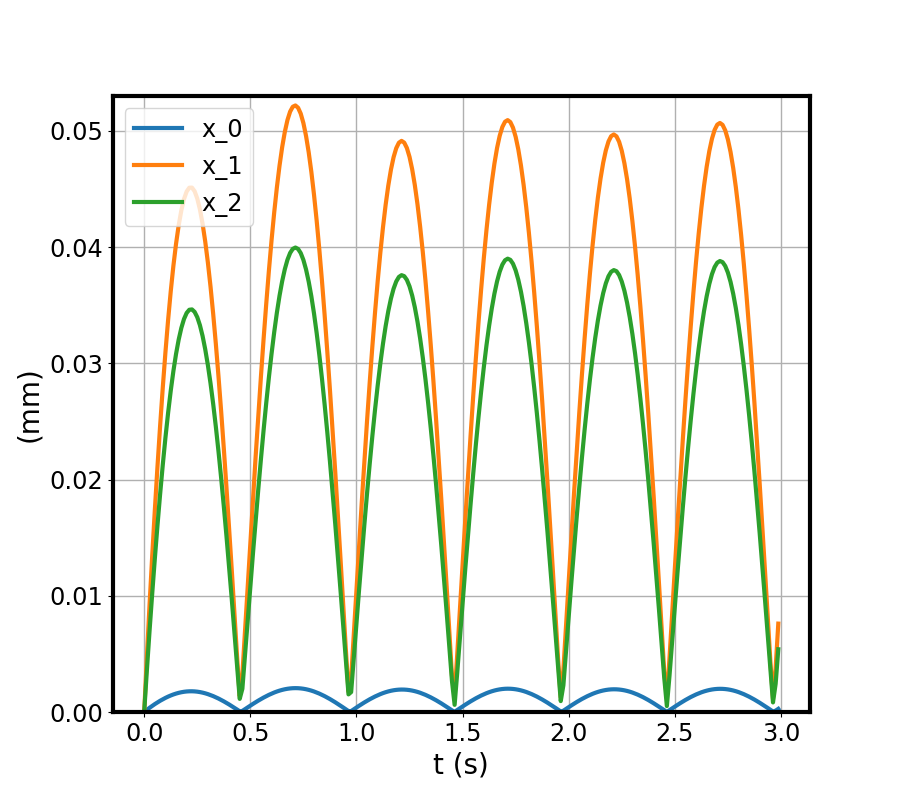} 
     \caption{'Standard' formulation}
   \end{subfigure}
     \caption{Comparison of displacements computed using the standard
       and total pressure formulation
       (cf.~Section~\ref{sec:numerics:brain}). Plots of displacement
       magnitude $|u(x_i, t)|$ versus time $t$, for a set of points
       $x_0, x_1, x_2$ (see Figure~\ref{fig:brain:mesh} for the
       location and precise coordinates of the points $x_i$): (a)
       Total-pressure mixed finite element formulation, (b) Standard
       mixed finite element formulation
       (cf.~Example~\ref{ex:mpet:standard}). The computed displacements
       clearly differ between the two solution methods.}
     \label{fig:brain:comp}
   \end{figure}
 \end{center}
  
\section{Conclusions}
\label{sec:conclusion}

In this paper, we have presented a new mixed finite element
formulation for the quasi-static multiple-network poroelasticity
equations. Our formulation introduces a single additional scalar field
unknown, the total pressure. We prove, via energy and semi-discrete
\emph{a priori} error estimates, that this formulation is robust in
the limits of incompressibility ($\lambda \rightarrow \infty$) and
vanishing storage coefficients ($c_j \rightarrow 0$), in contrast to
standard formulations. Finally, numerical experiments support the
theoretical results. For the numerical experiments presented here, we
have used direct linear solvers. In future work, we will address
iterative solvers and preconditioning of the MPET equations.

\FloatBarrier

\bibliographystyle{siamplain}
\bibliography{references}

\end{document}